\documentclass[11pt,reqno,oneside]{amsart}
\usepackage{amsmath,amssymb,amscd,amsthm,verbatim,xcolor}
\usepackage{amsbsy}

\usepackage{latexsym}
\usepackage{mathrsfs}
\usepackage{mathtools}
\usepackage{enumitem}

\usepackage{bbm}
\usepackage[margin=1in,marginparwidth=1in]{geometry}

\usepackage{graphicx}
\usepackage{cite}
\usepackage[
colorlinks,
linkcolor={black},
citecolor={black},
urlcolor={black}
]{hyperref}

\usepackage{tikz}
\usepackage{pgfplots}
\usepackage{tikzscale}
\usetikzlibrary{calc,matrix,fadings,fit,math,arrows.meta}
\pgfmathdeclarefunction{lambdaprime}{1}{%
	\pgfmathparse{sqrt(1-pow(#1,2))}
}
\pgfmathdeclarefunction{lyapexp}{2}{%
	\pgfmathparse{acos(pow(#1,2) * ((1+lambdaprime(#2))/(lambdaprime(#1)*#2)-(lambdaprime(#1)*#2)/(1+lambdaprime(#2)))/4)}
}
\def\tikzLyapExp#1#2{
	\begin{tikzpicture}[baseline=0,
		ac/.style={very thick,myblue,Parenthesis-Parenthesis},
		pp/.style={very thick,myorange,Bracket-Bracket},
		up/.style 2 args={domain={-lyapexp(#1,#2)+90}:{lyapexp(#1,#2)+90}},
		down/.style 2 args={domain={-lyapexp(#1,#2)-90}:{lyapexp(#1,#2)-90}},
		right/.style 2 args={domain={lyapexp(#1,#2)-90}:{-lyapexp(#1,#2)+90}},
		left/.style 2 args={domain={lyapexp(#1,#2)+90}:{-lyapexp(#1,#2)+270}}
		]
		\draw[ac,up=#1#2] plot ({cos(\x)},{sin(\x)});
		\draw[ac,down=#1#2] plot ({cos(\x)},{sin(\x)});
		\draw[pp,right=#1#2] plot ({cos(\x)},{sin(\x)});
		\draw[pp,left=#1#2] plot ({cos(\x)},{sin(\x)});
	\end{tikzpicture}
}

\definecolor{myblue}{RGB}{100,100,220}
\definecolor{myorange}{RGB}{255,200,100}

\newcommand{\bbC}{{\mathbb{C}}}
\newcommand{\bbD}{{\mathbb{D}}}

\newcommand{\bbM}{{\mathbb{M}}}
\newcommand{\bbN}{{\mathbb{N}}}
\newcommand{\bbQ}{{\mathbb{Q}}}
\newcommand{\bbR}{{\mathbb{R}}}
\newcommand{\bbS}{{\mathbb{S}}}
\newcommand{\bbT}{{\mathbb{T}}}
\newcommand{\bbU}{{\mathbb{U}}}
\newcommand{\bbZ}{{\mathbb{Z}}}

\def\idty{\mathbbm{1}}

\newcommand{\CE}{{\mathcal{E}}}

\newcommand{\CH}{{\mathcal{H}}}

\newcommand{\CL}{{\mathcal{L}}}
\newcommand{\CM}{{\mathcal{M}}}

\renewcommand{\Re}{{\mathrm{Re} \,}}
\renewcommand{\Im}{{\mathrm{Im} \,}}

\newcommand{\SL}{{\mathbb{SL}}}

\newcommand{\SU}{{\mathbb{SU}}}

\newcommand{\univ}{{\mathrm{univ}}}

\newcommand{\ac}{{\mathrm{ac}}}

\newcommand{\pp}{{\mathrm{pp}}}

\newcommand{\DC}{{\mathrm{DC}}}

\newtheorem{theorem}{Theorem}[section]
\newtheorem{lemma}[theorem]{Lemma}
\newtheorem{prop}[theorem]{Proposition}
\newtheorem{coro}[theorem]{Corollary}

\theoremstyle{definition}
\newtheorem{remark}[theorem]{Remark}
\newtheorem{definition}[theorem]{Definition}

\allowdisplaybreaks 
\numberwithin{equation}{section}

\DeclareMathOperator{\arcsinh}{arcsinh}

\newcommand{\reflected}{\mathscr{R}}
\newcommand{\costwopi}{{\mathrm{c}}}
\newcommand{\sintwopi}{{\mathrm{s}}}

\makeatletter
\def\subsection{\@startsection{subsection}{2}%
	\z@{.5\linespacing\@plus.7\linespacing}{.5\linespacing}%
	{\normalfont\scshape\centering}}
\makeatother

\hypersetup{
	pdftitle = {Exact Mobility Edges for almost-periodic CMV matrices via Gauge Symmetries},
	pdfauthor = {Christopher Cedzich,
		Jake Fillman,
		Long Li,
		Darren Ong,
		Qi Zhou},
	pdfsubject = {quantum walks, mobility edges, CMV matrices, Anderson localization, gauge symmetry},
	pdfkeywords = {quantum walks, mobility edges, CMV matrices, Anderson localization, gauge symmetry, almost-periodic cocycle, GECMV matrix, almost-Mathieu operator, unitary dynamics, global theory, spectral theory, dynamical systems, localization-delocalization transition, absolutely continuous spectrum, pure point spectrum
	}
}
\pgfplotsset{compat=1.18}

\begin{document}
	
	\title[CMV matrices with Mobility Edges]{Exact Mobility Edges for almost-periodic CMV matrices\\via Gauge Symmetries}
	
	\author[C.\ Cedzich]{Christopher Cedzich}
	\email{\href{mailto:cedzich@hhu.de}{cedzich@hhu.de}}
	\address{Quantum Technology Group, Heinrich Heine Universit\"at D\"usseldorf, Universit\"atsstr. 1, 40225 D\"usseldorf, Germany}
	
	\author[J.\ Fillman]{Jake Fillman}
	\email{\href{mailto:fillman@txstate.edu}{fillman@txstate.edu}}
	\address{Department of Mathematics, Texas State University, San Marcos, TX 78666, USA}
	
	\author[L.\ Li]{Long Li}
	\email{\href{mailto:ll106@rice.edu}{longli@rice.edu}}
	\address{Department of Mathematics, Rice University, Houston, TX 77005, USA}
	
	\author[D.\ C.\ Ong]{Darren C.\ Ong}
	\email{\href{mailto:darrenong@xmu.edu.my}{darrenong@xmu.edu.my}}
	\address{Department of Mathematics, Xiamen University Malaysia, Jalan Sunsuria, Bandar Sunsuria, 43900 Selangor, Malaysia}
	
	\author[Q.\ Zhou]{Qi Zhou}
	\email{\href{mailto:qizhou@nankai.edu.cn}{qizhou@nankai.edu.cn}}
	\address{Chern Institute of Mathematics and LPMC, Nankai University, Tianjin 300071, China}
	
	\renewcommand{\keywordsname}{Key words}
	\keywords{quantum walks, mobility edges, CMV matrices, gauge symmetry, Anderson localization}
	
	\begin{abstract}
		We investigate the symmetries of so-called generalized extended CMV matrices. It is well-documented that problems involving reflection symmetries of standard extended CMV matrices can be subtle. We show how to deal with this in an elegant fashion by passing to the class of generalized extended CMV matrices via explicit diagonal unitaries in the spirit of Cantero--Gr\"unbaum--Moral--Vel\'azquez.
		As an application of these ideas, we construct an explicit family of almost-periodic CMV matrices, which we call the mosaic unitary almost-Mathieu operator, and prove the occurrence of exact mobility edges. That is, we show the existence of energies that separate spectral regions with absolutely continuous and pure point spectrum and exactly calculate them.
	\end{abstract}
	
	\maketitle
	
	\setcounter{tocdepth}{1}
	\tableofcontents
	
	\hypersetup{
		linkcolor={black!30!blue},
		citecolor={black!30!blue},
		urlcolor={black!30!blue}
	}

	\section{Introduction}
	
	Cantero--Moral--Vel\'{a}zquez (CMV) matrices which arise in the study of orthogonal polynomials on the unit circle (OPUC), play a fundamental role in the spectral theory of unitary operators, analogous to the role played by Jacobi matrices and discrete Schr\"odinger operators in the theory of self-adjoint operators. For background, we direct the reader to the monographs \cite{Simon2005OPUC1, Simon2005OPUC2} and references therein. CMV matrices also play an important role in mathematical physics due to their connections with important models, notably, with quantum walks in one spatial dimension. Quantum walks, which function as quantum-mechanical analogs of classical random walks, are fundamental models in spectral theory and modern mathematical physics. Due to the fast spreading rate of quantum walks compared to classical random walks, they have shown promise in quantum algorithms \cite{santhaQuantumWalkBased2008, portugal2013quantum, Ambainis, SKW, gonzalezQuantumAlgorithmsPowering2021, apersUnifiedFrameworkQuantum2019} and quantum computing \cite{apersSimulationQuantumWalks2018, Venegas-Andraca, morioka2019detection, Lovett:2010ff, jefferyMultidimensionalQuantumWalks2022}. Additionally, they provide an excellent set of test cases to study discrete-time quantum dynamics \cite{SpaceRandom,SpacetimeRandom,grimmett2004weak,AVWW2011JMP,Joye_Merkli,aschLowerBoundsLocalisation2019} and model topological phases \cite{TopClass,Ti,WeAreSchur,AsboBB,KitaExploring,BBX,bourne2022index}. Quantum walks also represent a rich collection of objects on which one can study the interplay between spectral theory and discrete-time quantum dynamics \cite{ewalks,CFGW2020LMP}. 
	
	There is a mismatch between the two classes of objects which played a role in the work \cite{CFO1} and which we want to make explicit here. In the self-adjoint setting, the physical objects (discrete Schr\"odinger operators) comprise a subset of the collection of natural inverse spectral objects (Jacobi matrices); that is to say, every discrete Schr\"odinger operator is a Jacobi matrix. However, in the unitary setting, the situation is reversed: the inverse spectral objects (CMV matrices) comprise a subset of the physical objects (quantum walks). More precisely, a quantum walk has the form of a CMV matrix as long as the quantum coins have unit determinant and real and positive diagonal entries, which is not always a natural condition to impose on the associated physical system. In the present manuscript, we identify a \emph{split-step quantum walk} with an operator having the general appearance of an extended CMV matrix with complexified $\rho$'s; we called these \emph{generalized extended CMV matrices} (GECMV matrices) in \cite{CFO1} (see also \cite{bourgetSpectralAnalysisUnitary2003}). This additional freedom within the family of GECMV matrices is important; for example, it is what allowed the authors of \cite{CFO1} to make room for important techniques from the quasi-periodic theory including coupling constants, the Herman estimate, Aubry duality, and more. Also, it allowed for the introduction of randomly chosen phases in \cite{bourgetSpectralAnalysisUnitary2003} and the discussion of the density of states \cite{joyeDensityStatesThouless2004}, fractional moment estimates \cite{joyeFractionalMomentEstimates2005} and Anderson and dynamical localization in \cite{hamzaLocalizationRandomUnitary2006} and \cite{Hamza2009}, respectively.
	
	This mismatch between the physical and spectral objects has serious consequences: while the spectral theory of extended CMV matrices is well-developed \cite{Simon2005OPUC1, Simon2005OPUC2} with many useful tools such as subordinacy theory, Kotani theory, Avila's global theory, and others, less is known about the spectral theory of GECMV matrices and quantum walks. Some of these issues were dealt with by the authors of \cite{CFO1} in an ad-hoc manner. Thus, we seek to introduce suitable tools to establish the spectral theory of GECMV matrices in  a more systematic way, which is one motivation of our paper.
	
	Building on ideas of Cantero--Gr\"unbaum-Moral--Vel\'azquez \cite{canteroMatrixvaluedSzegoPolynomials2010}, we close this gap by show{ing} that any GECMV matrix can be transformed to a standard CMV matrix by a diagonal gauge. Moreover, there is a crucial point here: in the case of coins with unit determinant, we show that one can do this \emph{without} altering the Verblunsky coefficients. The ability to fix Verblunsky coefficients and vary other parameters within a family of GECMV matrices is important from the dynamical systems perspective, since, if the Verblunsky coefficients are dynamically defined over suitable base dynamics (e.g.\ a torus translation), then we can produce isospectral GECMV matrices that also fiber over the same base dynamics.
	
	Let us explain one way that we get some additional mileage out of the variation of the phases, beyond just showing that generalized CMV matrices are equivalent to ``standard'' CMV matrices. A technique that often plays a crucial role in the study of \emph{discrete Schr\"odinger operators} is the presence and use of suitable reflection symmetries. These symmetries are well-documented and manifest in a variety of ways, such as the symplectic symmetry of the associated transfer matrix cocycle. Indeed, these symmetries play a key role in, for instance, the study of localization with fixed frequency \cite{Jitomirskaya1999Annals, Jitomirskaya2012ADVANCES, JitomirskayaLiu2018ANNALS}. On the other hand, it is known that techniques centering on reflection symmetries of CMV matrices are more delicate, which makes it difficult, if not impossible, to study the mentioned localization phenomena strictly in this class of operators. However, in the class of GECMV matrices, one can work directly with operators having \emph{purely imaginary} $\rho$-values, which enjoy a particularly simple reflection symmetry. This observation and its application to models of interest seems to be new, as well, and holds promise for studying other aspects of quasi-periodic CMV matrices. Thus, we use the gauge freedom to pass to a GECMV matrix, reveal the hidden symmetry, and then use the gauge freedom again to deduce associated spectral consequences for the initial operator.
	
	Another motivation comes from two important topics in spectral theory: mobility edges and localization with fixed frequency. One of the most notable phenomena in spectral theory is the spectral phase transition between absolutely continuous and pure point spectral types as one varies a parameter within a given system. A significant instance of this phase transition occurs when several spectral types coexist simultaneously for the same operator, that is, the phase transition happens in the \emph{energy}. On account of the RAGE theorem, the quantum dynamics in the pure point part of the spectrum is localized whereas the dynamics in the absolutely continuous part of the spectrum exhibits transport in a suitable sense \cite{R, AG, E}. Thus, one refers to an energy separating pure point and absolutely continuous spectral regimes as a \emph{mobility edge}. Proving the existence of a mobility edge for multidimensional random operators remains a serious open problem in spectral theory and mathematical physics (compare \cite{Simon2000twentyfirst}). One of the most important families in which {spectral} phase transition{s} ha{ve} been observed is the \emph{almost-Mathieu operator}
	\begin{equation*}
		(H_{ V,\alpha,\theta}\,u)(n)=u(n+1)+u(n-1)+ V (n\alpha+\theta)u(n),
	\end{equation*} 
	where $V(x)=2\lambda \cos 2\pi x$ for $x \in \bbT :=\bbR/\bbZ$. The almost-Mathieu operator is known to exhibit phase transitions as the relevant parameters (coupling constant, frequency and phase) are varied \cite{AvilaAMOAC, AJ2012JEMS, AJZ2018MA, AYZ2017DUKE, Jitomirskaya1999Annals, Jitomirskaya2012ADVANCES, JitomirskayaLiu2018ANNALS}. Furthermore, the mosaic almost-Mathieu operator and the ``generalized'' André-Aubry model display exact mobility edges \cite{wangExactMobilityEdges2021a}. 
	Spectral phase transitions have been observed lately also in the \emph{unitary almost-Mathieu operator} (UAMO) \cite{CFO1}, which is defined as the GECMV matrix with Verblunsky coefficients
	\begin{align*}
		\alpha_{2n-1} &= \lambda_2 \sin(2\pi(n\Phi+\theta)),  & \alpha_{2n} &= \lambda_1',\\  
		\rho_{2n-1} &=  \lambda_2 \cos(2\pi(n\Phi+\theta)) - i \lambda_2', 	 & \rho_{2n} &= \lambda_1,
	\end{align*}
	where $\lambda_i\in[0,1]$ and $\lambda_i'=\sqrt{1-\lambda_i^2}$.
	However, it is unknown whether there exist  GECMV matrices and extended CMV matrices which have exact mobility edges.

	Specifically, to establish the presence of mobility edges for extended CMV matrices, the key issue is to obtain Anderson localization for fixed frequency, since establishing the presence of purely absolutely continuous spectrum is well-developed for quasi-periodic extended CMV matrices \cite{LDZ2022TAMS}.  We should point out that in the Schr\"odinger case, Anderson localization for fixed frequency is quite an important issue in establishing the Ten Martini Problem \cite{AJ09}, the universal  hierarchical structure of eigenfunctions \cite{JitomirskayaLiu2018ANNALS} and the sharp arithmetic phase transition \cite{jitomirskaya2018universal}. In the quasi-periodic extended CMV setting and in the positive Lyapunov exponent regime, Anderson localization with fixed phase is given by Damanik-Wang \cite{WangDamanik2019JFA} (in the same spirit as in Bourgain-Goldstein \cite{bourgain2000nonperturbative}). However, it is still a major challenge to establish Anderson localization for fixed Diophantine frequency for general almost-periodic extended CMV matrices.\footnote{See, however, \cite{yangLocalizationMagneticQuantum2022} for a result in the case of specific  \emph{generalized} extended CMV.} Our main results give a profitable step forward and a new set of tools in this regard.
	
	\begin{figure}[t]
		\begin{center}
			\includegraphics[width=.85\textwidth]{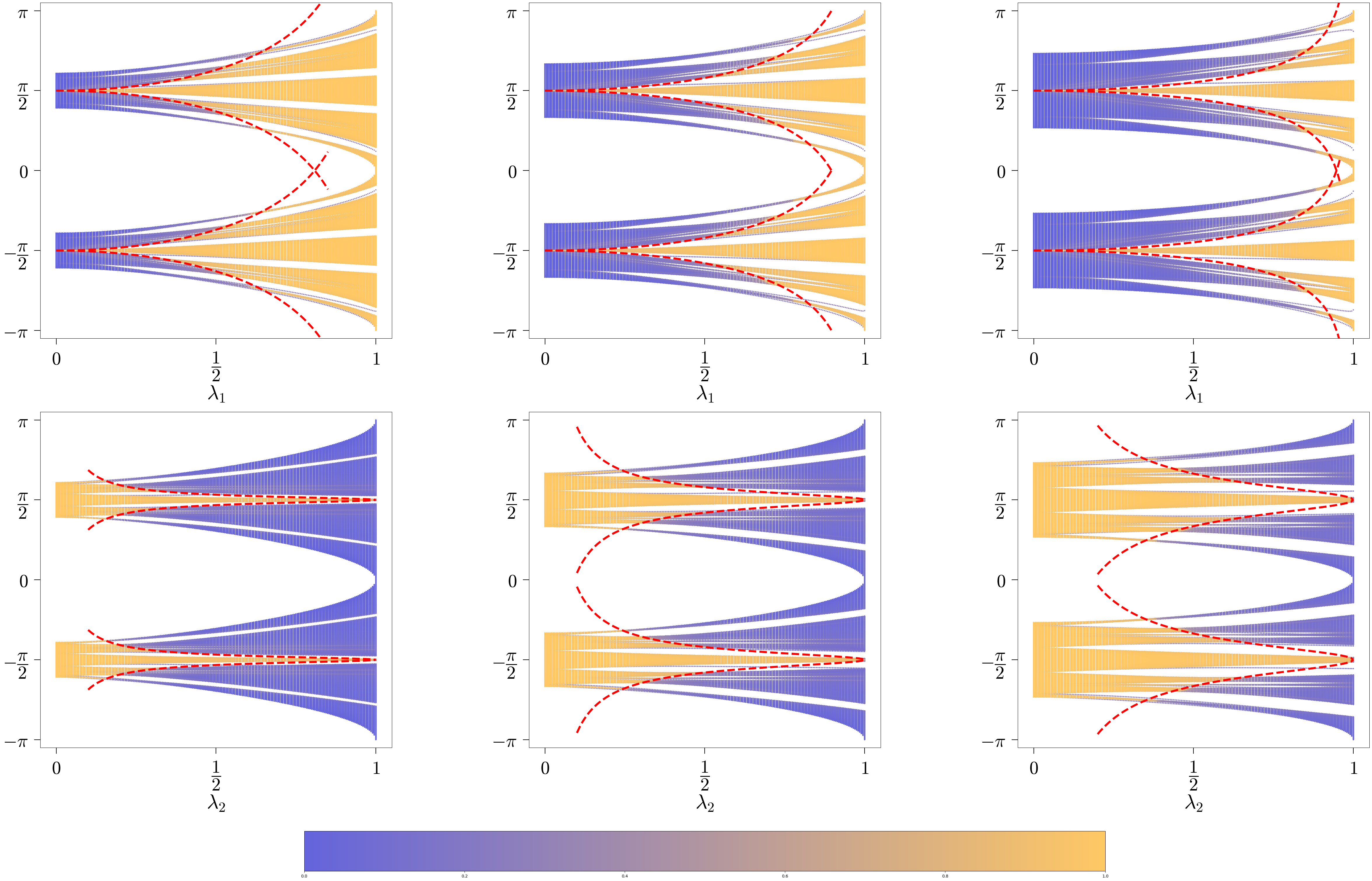}
		\end{center}
		\caption{\label{fig:mobility_edges}Mobility edges for the mosaic UAMO for $\Phi=(\sqrt5-1)/2$. In the upper row we vary $\lambda_1$ while keeping fixed to $\lambda_2\in\{1/3,1/2,2/3\}$, whereas in the bottom row we vary $\lambda_2$ and fix $\lambda_1\in\{1/3,1/2,2/3\}$. The dashed red lines correspond to $\pm\pi/2\cos(t_0)$ with $t_0$ given in \eqref{eq:t_0}. The coloring encodes the so-called \emph{fractal dimension} $\Gamma(m)$ which measures how spread out a generalized eigenfunction is. 
			As a rule of thumb, the more localized the generalized eigenfunctions are at $z$, the smaller its fractal dimension. 
			It is helpful to consider the limiting cases:
			If $z$ is a proper eigenvalue, then $\Gamma(z)=0$. On the other hand, if $z$ gives a plane wave solution whose mass is equally distributed on all sites, then $\Gamma(z)=1$. For more background on the fractal dimension, see e.g. \cite{bellAtomicVibrationsVitreous1970,thoulessElectronsDisorderedSystems1974}.
		}
	\end{figure}
	
	In this manuscript, we construct a family of GECMV matrices that is derived from a quantum walk with quasi-periodic coins which are periodically inserted into an otherwise fully transmitting medium. Using the ideas discussed above, we prove an exact mobility edge result in the case in which every other coin is generated by the quasi-periodic sequence, which we call the mosaic UAMO (see Section~\ref{sec:results} for detailed definitions and Section \ref{sec:MUAMO} for the physical background). The idea of potentials taking different values at even and odd sites as in the mosaic UAMO has a natural physical background. For example, it appeared in the study of the classical Su--Schrieffer--Heeger (SSH) model \cite{heeger1988jr} and driven conformal field theory \cite{wen2021periodically}. Recently, the quasi-periodic mosaic model \cite{wangOneDimensionalQuasiperiodicMosaic2020, wangExactMobilityEdges2021a} has been experimentally realized to detect exact mobility edges \cite{gao2023experimental}.
	
	The remainder of the paper is structured as follows: In the next section we introduce the model(s) we consider, that is, the GECMV matrices and the mosaic UAMO as a special case thereof, and state our main results. Section \ref{sec:MUAMO} provides the physical background on the mosaic UAMO model. In Section \ref{sec:GECMV} we prove the main structural result relating different GECMV matrices and discuss their symmetries. In Section \ref{sec:lyap} we classify the cocycles corresponding to the mosaic UAMO and calculate its Lyapunov exponent. In Sections \ref{sec:localization} and \ref{sec:subcriticalAc} we prove the exact mobility edges for the mosaic UAMO by showing that the spectral type is pure point and absolutely continuous in the super- and subcritical regime, respectively.

	\subsection*{Acknowledgements}  D.~C.~O.\ was supported in part by a grant from the Fundamental Research Grant Scheme from the Malaysian Ministry of Education (grant number FRGS/1/2022/TK07/XMU/01/1), a grant from the National Natural Science Foundation of China  (grant number 12201524), and a Xiamen University Malaysia Research Fund (grant number XMUMRF/2023C11/IMAT/0024). C.~C.\ was supported in part by the Deutsche Forschungsgemeinschaft (DFG, German Research Foundation) under the grant number 441423094. J.\ F.\ was supported in part by National Science Foundation grant DMS--2213196 and Simons Foundation Collaboration Grant \#711663.  Q.~Z.\ was partially supported by National Key R\&D Program of China (2020YFA0713300)  NSFC grant (12071232), the Science Fund for Distinguished Young Scholars of Tianjin (No. 19JCJQJC61300) and Nankai Zhide Foundation.
	
	\section{Model and Results} \label{sec:results}
	
	\subsection{Generalized Extended CMV Matrices}
	
	Consider the Hilbert space $\CH:=\ell^2(\bbZ)$ with the standard basis $\{\delta_n : n \in \bbZ\}$. On $\CH$, we consider \emph{generalized extended CMV matrices} $\mathcal{E} = \mathcal{E}(\alpha,\rho)$ defined by $\mathcal{E}=\mathcal{L}\mathcal{M}$, where $\mathcal{L}=\bigoplus_{n\in\bbZ}\Theta(\alpha_{2n},\rho_{2n})$ and $\mathcal{M}=\bigoplus_{n\in\bbZ}\Theta(\alpha_{2n+1},\rho_{2n+1})$ are specified by
	\begin{equation}\label{eq:theta_mat}
		\Theta(\alpha,\rho)=\begin{bmatrix}\overline{\alpha}&\rho\\\overline{\rho}&-\alpha\end{bmatrix}
	\end{equation}
	with Verblunsky pairs
	\begin{equation}\label{eq:verblunsky_pair}
		(\alpha,\rho)\in\bbS^3=\{(z_1,z_2)\in\overline\bbD^2:|z_1|^2+|z_2|^2=1\}.
	\end{equation}
	In the definitions of $\CL$ and $\CM$, we note that $\Theta(\alpha_j,\rho_j)$ acts on the subspace $\ell^2(\{j,j+1\})$. In the standard basis $\{\delta_n : n \in \bbZ\}$ of $\ell^2(\bbZ)$, such a GECMV matrix takes the form
	\begin{equation} \label{eq:gecmv}
		\mathcal E 	= 
		\begin{bmatrix}
			\ddots & \ddots & \ddots & \ddots &&&&  \\
			& \overline{\alpha_0\rho_{-1}} & \boxed{-\overline{\alpha_0}\alpha_{-1}} & \overline{\alpha_1}\rho_0 & \rho_1\rho_0 &&&  \\
			& \overline{\rho_0\rho_{-1}} & -\overline{\rho_0}\alpha_{-1} & {-\overline{\alpha_1}\alpha_0} & -\rho_1 \alpha_0 &&&  \\
			&&  & \overline{\alpha_2\rho_1} & -\overline{\alpha_2}\alpha_1 & \overline{\alpha_3} \rho_2 & \rho_3\rho_2 & \\
			&& & \overline{\rho_2\rho_1} & -\overline{\rho_2}\alpha_1 & -\overline{\alpha_3}\alpha_2 & -\rho_3\alpha_2 &    \\
			&& && \ddots & \ddots & \ddots & \ddots &
		\end{bmatrix},
	\end{equation}
	where we boxed the $(0,0)$ matrix element of $\CE$.

	``Generalized'' here means that the $\rho$'s are allowed to take complex values inside the closed unit disk $\overline\bbD$, in contrast to \emph{standard extended CMV matrices} as defined in \cite{canteroFivediagonalMatricesZeros2003}, where the $\rho$'s merely take real values in $(0,1]$. Let us mention that this complexification of extended CMV matrices was originally motivated by physical models \cite{blatterZenerTunnelingLocalization1988,bourgetSpectralAnalysisUnitary2003}. Moreover, this class of operators is motivated by the study of split-step quantum walks whose quantum coins have determinant one; indeed, if one takes such a split step walk and writes down the matrix with respect to the ordered basis 
	\begin{equation*} 
		\ldots \delta_{-1}^-, \delta_0^+, \delta_0^-, \delta_1^+, \ldots,
	\end{equation*}    
	then the associated matrix is exactly a GECMV matrix with suitable $(\alpha,\rho)$; see Section \ref{sec:MUAMO}. As discussed in the introduction, this generalization turned out to be essential to the work \cite{CFO1} since the complexification of the $\rho$ parameters (which was motivated by the choice of magnetic translations for an associated 2D model) was absolutely crucial to make room for the magic of duality, the Herman estimate, and other techniques.
	
	To study the spectral properties of $\mathcal{E}$, one naturally considers the generalized eigenvalue equation $\mathcal{E}u=zu$ for $z\in\bbC$. Solutions to this equation satisfy the iterative relation
	\begin{equation*}
		\begin{bmatrix}u_{2n+1}\\u_{2n}\end{bmatrix}=A_{n,z}\begin{bmatrix}u_{2n-1}\\u_{2n-2}\end{bmatrix}, \quad n \in \bbZ,
	\end{equation*}
	where the \emph{transfer matrices} $A_{n,z}$ are given by
	\begin{equation}\label{eq:GECMV_transmat}
		A_{n,z}=\frac{1}{\rho_{2n}\rho_{2n-1}}\begin{bmatrix}
			z^{-1}+\alpha_{2n}\overline{\alpha}_{2n-1}+\alpha_{2n-1}\overline{\alpha}_{2n-2}+\alpha_{2n}\overline{\alpha}_{2n-2}z& -\overline{\rho}_{2n-2}\alpha_{2n-1}-\overline{\rho}_{2n-2}\alpha_{2n}z\\
			-\rho_{2n}\overline{\alpha}_{2n-1}-\rho_{2n}\overline{\alpha}_{2n-2}z&\rho_{2n}\overline{\rho}_{2n-2}z
		\end{bmatrix},
	\end{equation}
	for $n \in \bbZ$ and $z \in \bbC \setminus \{0\}$. This follows from direct calculations, which are carried out in detail in \cite[Section~4]{CFO1}.
	
	We will relate isospectral families of GECMV matrices at two levels: First, we show that any two GECMV matrices with the same $\alpha$'s are unitarily equivalent via a diagonal unitary. Thus, the spectral type and properties of solutions to the eigenvalue equation are \emph{independent} of the phase of the $\rho$'s. Second, later in the paper, we show how to relate the transfer matrix cocycle as in \eqref{eq:GECMV_transmat} to the the Szeg\H{o} \cite{Simon2005OPUC1} cocycle. These are also related to the Gesztesy--Zinchenko \cite{gesztesyWeylTitchmarshTheory2006} cocycle via an identity elucidated in \cite{DFO2016JMPA}, but we will not need that connection here. We give precise definitions of these objects later. We anticipate that these ideas and connections will be useful in other contexts.
	
	\begin{theorem}\label{thm:rho_phases}
		Any two GECMV matrices with the same $\alpha$'s are unitarily equivalent, and thus isospectral. More precisely, given a set of Verblunsky coefficients $\{\alpha_{n}:n\in\bbZ\}\subset\bbD$, let $\rho_n=\sqrt{1-|\alpha_n|^2}$. Then, for any two sequences $\{\xi_{n}\}_n,\{\zeta_{n}\}_n\subset \partial \bbD$, the GECMVs $\mathcal E^\xi$ and $\mathcal E^\zeta$ associated to coefficient sequences $\{\alpha_{n},\xi_{n}\rho_n\}$ and $\{\alpha_{n},\zeta_{n}\rho_n\}$, respectively, are gauge equivalent, i.e., there exists a diagonal unitary matrix $D$ so that $\mathcal E^\xi=D^*\mathcal E^\zeta D$.  
	\end{theorem}

	\begin{remark}
		Verblunsky's Theorem (also called Favard's Theorem on the circle, compare, \cite[Section~1.1]{Simon2005OPUC1}) sets up a one-to-one correspondence $\mu\leftrightarrow\{\alpha_{n}\}_{n=0}^{\infty}$ between non-trivial probability measures on the unit circle $\partial\bbD$ and $\bigtimes_{j=0}^{\infty}\bbD$. This correspondence does not care about the values of $\rho_{n}$'s. Theorem~\ref{thm:rho_phases}  shows the isospectral nature of GECMVs associated to the \emph{phased} $\rho_{n}$'s.
	\end{remark}
	
	This immediately implies that every GECMV matrix can be turned into a standard extended CMV matrix. This requires transforming the $\rho$'s to nonnegative real numbers, which can be achieved via a diagonal gauge transformation, and thus generalizes the technique in \cite[Sect.~7]{canteroMatrixvaluedSzegoPolynomials2010}:
	\begin{coro}\label{cor:CMV_realified}
		Every GECMV matrix  is gauge-equivalent to a standard extended CMV matrix. More precisely, for any GECMV matrix $\mathcal E$ determined by Verblunsky pairs $(\alpha_k,\rho_k)_{k\in\bbZ}$, there is a diagonal unitary operator $D$ such that $D^*\mathcal ED$ is a standard extended CMV matrix with Verblunsky pairs $(\alpha_k,|\rho_k|)_{k\in\bbZ}$.
	\end{coro}

	\begin{remark} \label{rem:gecmvWarning}
		The matrix form of $\mathcal E$ given in \eqref{eq:gecmv} and the condition \eqref{eq:verblunsky_pair} are essential for the the ability to choose the diagonal conjugation in such a way that the $\alpha$'s remain fixed. As discussed above, this corresponds to split-step walks with unimodular coins. In the more general setting, one is led to operators such that $|\alpha|^2+|\rho|^2\in \partial \bbD$; see Appendix~\ref{sec:supergecmv} for details. Here, one has to be slightly more careful, but the basic idea still works; compare \cite{ewalks}. As a word of warning, however, it is sometimes not possible to choose the gauge in such a manner as to fix the $\alpha$'s in this general setting, and in particular, the \emph{base dynamics} may no longer be gauge invariant within the class of isospectral GECMV. For instance, the CMV matrix corresponding to the quasi-periodic quantum walk in \cite{ewalks} is not quasi-periodic anymore; instead, its Verblunsky coefficients are generated by the skew-shift.
	\end{remark}
	
	As discussed above, one of the pleasant outcomes of this approach is that it enables us to deal with reflection symmetries in a useful way. See Section~\ref{sec:reflection} for detailed statements, and note that the desired reflection symmetry for the $\rho$ terms is given by \eqref{eq.newReflectionProp}, which forces one to consider $\rho$ values outside of $[0,1]$. We anticipate that this perspective will lead to useful results in other contexts.
	
	\subsection{Almost-periodic GECMV matrices}
	
	Our work is motivated by the study of certain almost-periodic (but not quasi-periodic) quantum walks, which lead to the following GECMV matrices. We will explain the origin of this model in Section \ref{sec:MUAMO}.  We here consider a model where all even Verblunsky pairs are constant, every $s$-th odd one is given by a quasi-periodic function, and all others are ``trivial'' from a dynamical perspective. Concretely, let $\theta,\Phi\in\mathbb{T}$, and $\lambda_1,\lambda_2\in[0,1]$, and consider
	\begin{equation}\label{eq:mosaicVerblunskieS}
		\begin{alignedat}{4}
			(\alpha_{2n-1},\rho_{2n-1})&=\begin{cases}
				(\lambda_{2}\sin2\pi(\theta+n\Phi),\lambda_{2}\cos2\pi(\theta+n\Phi)-i\lambda_{2}'),&	n\in s\bbZ,\\
				(0,-i),&	n\in \bbZ\backslash s\bbZ,
			\end{cases}\\
			(\alpha_{2n},\rho_{2n})&=(\lambda_1',\lambda_1),\qquad n\in\bbZ,
		\end{alignedat}
	\end{equation}
	where $\lambda_{i}'=\sqrt{1-\lambda_{i}^{2}},i=1,2$.
	The case $s=1$ corresponds to the \emph{unitary almost-Mathieu operator (UAMO)} and was studied extensively in \cite{CFO1}. 
	For reasons that will become clear later, we will call GECMV matrices with coefficients as in \eqref{eq:mosaicVerblunskies} the \emph{mosaic UAMO} (see Section~\ref{sec:MUAMO}) and denote them by $\mathcal E_{\Phi,\lambda_1,\lambda_2,s}(\theta)$ or $\mathcal E(\theta)$ for short when all parameters are fixed.
	
	By a well-known argument using minimality of $\theta \mapsto \theta + \Phi$ on $\bbT := \bbR/\bbZ$ and strong operator approximation, there is a fixed set $\Sigma_{\lambda_1,\lambda_2,\Phi,s}$ such that (compare \cite[Theorems~10.9.13 and 10.9.14]{Simon2005OPUC2} for a discussion in the case of standard (half-line) CMV matrices and \cite[Theorem 4.9.1]{damanik2022one} for a proof in the case of discrete Schr\"odinger operators)
	\begin{equation*}
		\sigma(\mathcal E_{\lambda_1,\lambda_2,\Phi,\theta,s}) = \Sigma_{\lambda_1,\lambda_2,\Phi,s} \quad \forall \theta \in \bbT.
	\end{equation*}
	
	For physical reasons given in \cite[Section 3]{CFO1}, we call $\lambda_1,\lambda_2\in[0,1]$ ``coupling constants'', $\Phi\in\bbT$ the ``frequency'' and $\theta\in\bbT$ the ``phase''.
	The arithmetic properties of $\Phi$ play a crucial role in determining spectral properties of the underlying operator. We call $\Phi$ \emph{Diophantine} if there exist $\kappa>0,\tau>1$ such that
	\begin{equation}\label{eq:diophsigmagamma}
		\|n \Phi\|_{\bbT}:= \inf_{p\in\bbZ} |n \Phi-p|  \geq\frac\kappa{|n|^\tau}\quad\forall n\neq0.
	\end{equation}
	In this case, we write $\Phi\in \DC(\kappa,\tau)$. Moreover, we shall denote the set of all Diophantine frequencies by
	\begin{equation}\label{eq:dioph}
		\DC=\bigcup_{\kappa>0,\tau>1}\DC(\kappa,\tau).
	\end{equation}
	
	We are mostly interested in the simplest non-trivial case $s=2$ for which \eqref{eq:mosaicVerblunskieS} reduces to
	\begin{equation}\label{eq:mosaicVerblunskies}
		\begin{alignedat}{4}
			\alpha_{4n-1}&=\lambda_{2}\sin2\pi(\theta+2n\Phi),&\qquad\alpha_{4n+1}&=0,&\qquad\alpha_{4n}&=\alpha_{4n+2}=\lambda_{1}',\\
			\rho_{4n-1}&=\lambda_{2}\cos2\pi(\theta+2n\Phi)-i\lambda_{2}',&\qquad\rho_{4n+1}&=-i,&\qquad\rho_{4n}&=\rho_{4n+2}=\lambda_{1}.
		\end{alignedat}
	\end{equation}
	In order to compactly refer to arcs on the circle $\partial \bbD$, we write $(\zeta_1,\zeta_2)$ to denote the open arc of $\partial \bbD$ from $\zeta_1$ to $\zeta_2$ in the positive (counterclockwise) direction. The following {establishes the presence of exact mobility edges for the mosaic UAMO and} is one of the main results of this paper:
	
	\def\param{0.6}
	\def\paramm{0.7}
	\def\parammm{0.8}
	
	\begin{figure}[t]
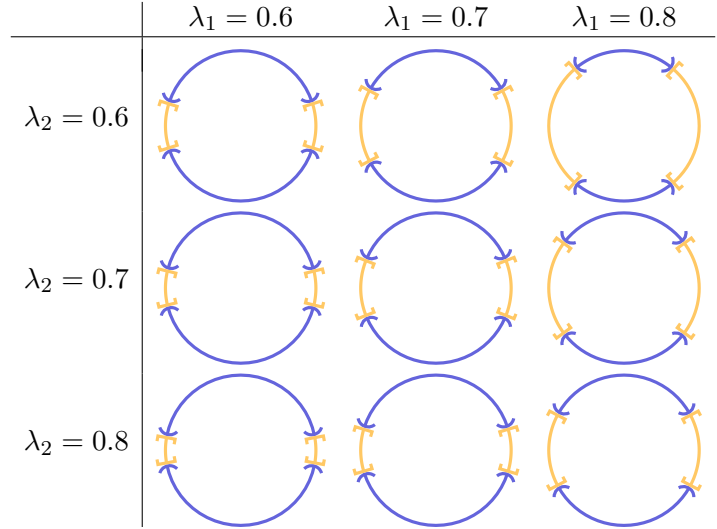

		\begin{tabular}{c|ccc}
			&$\lambda_1=\param$&$\lambda_1=\paramm$&$\lambda_1=\parammm$\\\hline\\[-.3cm]
			$\lambda_2=\param$
			&
			\tikzLyapExp\param\param
			&
			\tikzLyapExp\paramm\param
			&
			\tikzLyapExp\parammm\param
			\\[1cm]
			$\lambda_2=\paramm$
			&
			\tikzLyapExp\param\paramm
			&
			\tikzLyapExp\paramm\paramm
			&
			\tikzLyapExp\parammm\paramm
			\\[1cm]
			$\lambda_2=\parammm$
			&
			\tikzLyapExp\param\parammm
			&
			\tikzLyapExp\paramm\parammm
			&
			\tikzLyapExp\parammm\parammm
		\end{tabular}
		\caption{``Phase diagram'' of the mosaic UAMO with the mobility edges determined in Theorem \ref{thm:mobs} plotted for various $\lambda_1,\lambda_2\in\{\param,\paramm,\parammm\}$. The blue arcs contain the absolutely continuous spectrum, and the yellow arcs contain the pure point spectrum. \label{fig:phasediagram} } 
	\end{figure}

	\begin{theorem}\label{thm:mobs}
		Fix $\Phi \in \DC$ and $\lambda_1,\lambda_2 \in (0,1)$, satisfying 
		\begin{equation} \label{eq:lambdacond} 
			\frac{\lambda_1^2}{\lambda_1'} < \frac{2\lambda_2}{\lambda_2'}. 
		\end{equation}
		For each $\theta \in \bbT$, define $\alpha = \alpha(\theta)$ and $\rho = \rho(\theta)$ by \eqref{eq:mosaicVerblunskies} and consider the associated GECMV matrix $\mathcal{E}(\theta):= \mathcal{E}(\alpha(\theta), \rho(\theta))$ as in \eqref{eq:gecmv}. 	
		Choose $t_0 \in (0,\pi/2)$ 
		such that
		\begin{equation} \label{eq:t0DEF}	
			\cos(t_{0})=\frac{\lambda_{1}^{2}\lambda_{2}'}{2\lambda_{1}'\lambda_{2}}.
		\end{equation}
		Then, for any $\xi=\{\xi_n:n\in\bbZ\}\subset\partial\bbD$,
		\begin{enumerate}[label={\rm (\alph*)}, ref={\thetheorem\ (\alph*)}, itemsep=1ex]
			\item \label{thm:mobs:ac}
			$\mathcal{E}^\xi(\theta)$ has purely absolutely continuous spectrum in
			\begin{equation*}
				I_\ac := (e^{-it_0+\pi/2},e^{it_0+\pi/2})\cup(e^{-it_0-\pi/2},e^{it_0-\pi/2})
			\end{equation*} 
			for every $\theta \in \bbT$. 
			\item \label{thm:mobs:pp}
			$\mathcal{E}^\xi(\theta)$ exhibits Anderson localization in
			\begin{equation*}
				I_\pp := (e^{it_0 - \pi/2}, e^{-it_0+\pi/2}) \cup (e^{it_0+\pi/2},e^{-it_0-\pi/2})
			\end{equation*}
			for every $\theta$ that is non-resonant with respect to $\Phi$ {\rm(}in particular, for a.e.\ $\theta${\rm)}.
		\end{enumerate}
	\end{theorem}

	\begin{remark} \mbox{}
		\begin{enumerate}[itemsep=1ex]
			\item To the best of our knowledge, this gives the first explicit example of almost-periodic (GE)CMV matrices/quantum walks enjoying an exact mobility edge.  Part (b) is of particular interest. Recalling Corollary \ref{cor:CMV_realified}, this gives the first family of almost-periodic extended CMV matrices that has Anderson localization for fixed frequency.
			\item The condition on the coupling constant $\lambda_1$ and $\lambda_2$ in the statement of the theorem ensures that there is a genuine mobility edge, i.e. that $t_0$ is well-defined.
			\item The reader may consult Figure~\ref{fig:phasediagram} for an illustration of the different spectral regions for varying coupling constants, and Figure \ref{fig:mobility_edges} for numerical simulation thereof.
			\item We compute exactly the Lyapunov exponent on the spectrum (see Theorem~\ref{Lyapunov} for a detailed statement) and show that the eigenfunctions decay at the Lyapunov rate (see Theorem~\ref{t:efctDecay})
		\end{enumerate}
	\end{remark}

	\section{Physical Motivation: The Mosaic UAMO}\label{sec:MUAMO}
	
	Let us describe the physical model that motivates both our study of GECMV matrices and of mobility edges for the mosaic UAMO, that is, one-dimensional quantum walks of split-step type. These systems are specified by a generalized shift and a coin sequence, which for the mosaic UAMO  alternates between a quasi-periodic coin and $s-1$ perfectly transmitting coins. The generalized shift as well as the quasi-periodic coin sequences are derived from the \emph{unitary almost-Mathieu operator} (UAMO) \cite{CFO1} which describes the discrete time evolution of a particle on two-dimensional lattice with perpendicular magnetic field \cite{CFGW2020LMP,CGWW2019JMP}.
	
	Let $\mathscr{H}=\ell^{2}(\bbZ)\otimes \bbC^{2}$. On this Hilbert space, we consider a \emph{split-step quantum walk} $W:\mathscr{H}\to\mathscr{H}$ that is given as a product of a conditional shift operator that additionally depends on a \emph{coupling constant} and a coin operator. To define these operators, let us write the standard basis of $\mathscr{H}$ as 
	\begin{equation}\label{eq:base}
		\delta_{n}^{s}=\delta_{n}\otimes e_{s},\quad n\in\bbZ,s\in\{+,-\},
	\end{equation}	
	where $\{\delta_{n}:n\in\bbZ\}$ is the standard basis of $\ell^{2}(\bbZ)$ and $\{e_{+}=(1,0)^{\top}, \ e_{-}=(0,1)^{\top}\}$ is the standard basis of $\bbC^{2}$.
	With respect to this basis, we denote the coordinates of an element $\psi\in\mathscr H$ as $\psi_n^s=\langle\delta^s_n,\psi\rangle$, so that $\psi=\sum\nolimits_{n\in\bbZ}\psi_{n}^{+}\delta_{n}^{+}+\psi_{n}^{-}\delta_{n}^{-}$.
	
	The conditional shift operator with coupling constant $\lambda \in [0,1]$ is specified by its action on basis elements as
	\begin{equation*}
		S_{\lambda}\delta_{n}^{\pm}=\lambda\delta_{n\pm 1}^{\pm}\pm\lambda'\delta_{n}^{\mp},\quad \lambda'=\sqrt{1-\lambda^{2}}
	\end{equation*}
	and a coin operator $Q$ which acts coordinatewise via a sequence of local unitary coins
	\begin{equation*}
		Q_{n}=\begin{bmatrix}q_{n}^{11}&q_{n}^{12}\\q_{n}^{21}&q_{n}^{22}\end{bmatrix}\in \bbU(2,\bbC), \quad n \in \bbZ,
	\end{equation*}
	that is, $[Q\psi]_n=Q_n\psi_n$, where $\psi_{n}=[\psi_{n}^{+},\psi_{n}^{-}]^{\top}$. With these definitions, the split-step walk operator $W$ is given by
	\begin{equation}\label{eq:RandWalk}
		W=S_{\lambda}Q.
	\end{equation}
	Identifying $\ell^2(\bbZ)\otimes\bbC^2$ with $\ell^2(\bbZ)$ by ordering the basis in \eqref{eq:base} as \begin{equation}
		\dots,\delta_{-1}^-,\delta_{0}^+,\delta_{0}^-,\delta_{1}^+,\delta_{1}^-,\delta_{2}^+,\dots,
	\end{equation}
	we may identify the split-step walk $W$ defined in \eqref{eq:RandWalk} with the GECMV matrix $\CE$ with Verblunsky parameters by setting
	\begin{equation}\label{GECMV}
		Q_{n}=:\begin{bmatrix}
			\overline{\rho_{2n-1}}&-\alpha_{2n-1}\\
			\overline{\alpha_{2n-1}}&\rho_{2n-1}
		\end{bmatrix}, \qquad (\alpha_{2n},\rho_{2n}):=(\lambda',\lambda).
	\end{equation}
	
	Using this connection, the authors of \cite{CFO1} introduced a new coupling constant in the definition of quasi-periodic coin sequences to create room for the magic of Andr\'{e}--Aubry duality. More specifically, in \cite{CFO1} the local coins $Q_{n}$ are generated in a dynamical way as
	\begin{equation}\label{eq:UAMO_coin}
		Q_{n}=Q_{n,\lambda_2,\Phi,\theta}=\begin{bmatrix}\lambda_{2}\cos(2\pi(\theta+n\Phi))+i\lambda_{2}'&-\lambda_{2}\sin(2\pi(\theta+n\Phi))\\\lambda_{2}\sin2\pi(\theta+n\Phi)&\lambda_{2}\cos2\pi(\theta+n\Phi)-i\lambda_{2}'\end{bmatrix}
	\end{equation}
	where $\lambda_{2}\in[0,1]$, $\lambda_{2}'=\sqrt{1-\lambda_{2}^{2}}$, $\Phi\in\bbT:=\bbR/\bbZ$ is the {\it frequency} and $\theta\in\bbT$ is the {\it phase}. 
	The constant appearing in the shift in \eqref{eq:RandWalk} will be denoted as $\lambda_{1}$ and the shift operator will accordingly be denoted by $S_{\lambda_{1}}$. The resulting quantum walk $W_{\lambda_{1},\lambda_{2},\Phi,\theta}$ was dubbed the {\it unitary almost-Mathieu oparator} (UAMO) in \cite{CFO1} due to the close parallels between this model and the almost-Mathieu operator (AMO).

	In the same spirit, plugging the Verblunsky coefficients from \eqref{eq:mosaicVerblunskies} into \eqref{GECMV} identifies the GECMV matrix $\mathcal E(\theta)$ defined in \eqref{eq:mosaicVerblunskies} as a \emph{mosaic model} derived from the UAMO with local coins determined by
	\begin{equation}\label{eq:mosaic_coins}
		\begin{aligned}
			Q_n 
			&=\begin{dcases}\begin{bmatrix}\lambda_2\cos(2\pi(n\Phi+\theta))+i\lambda_2'&-\lambda_2\sin(2\pi(n\Phi+\theta))\\ 
					\lambda_2 \sin(2\pi(n\Phi+\theta)) & \lambda_2\cos(2\pi(n\Phi+\theta)) - i\lambda_2'\end{bmatrix} & n \in s\bbZ \\[5mm] \begin{bmatrix}i&0\\0&-i\end{bmatrix} & n \in \bbZ \setminus s \bbZ.\end{dcases}
		\end{aligned}
	\end{equation}
	Here, $s \geq 1$ is a fixed integer that determines the ``step size'': every $s$-th coin is the same as in the UAMO, and all others are replaced by perfectly transmitting coins. More precisely, we set $\lambda_2=0$ at lattice sites $n\notin s\bbZ = \{s m : m \in \bbZ\}$.

	The resulting walk $W = S_{\lambda_1}Q$ with coin operator $Q = Q_{\lambda_2,\Phi,\theta,s}$ corresponding to the sequence of local coins defined in \eqref{eq:mosaic_coins}  will be denoted $W_{\lambda_1,\lambda_2,\Phi,\theta,s}$. This model can be thought of as a unitary analogue of the almost-periodic mosaic model studied in \cite{wangOneDimensionalQuasiperiodicMosaic2020, wangExactMobilityEdges2021a}, that is, the discrete Schr\"odinger operator  $H_{V,\Phi,\theta}$  with onsite potential 
	\begin{equation*}
		V_n  
		=\begin{cases} 2\lambda \cos(2\pi(n\Phi+\theta)),&  n \in s\bbZ \\
			0,& n \in \bbZ \setminus s \bbZ.
		\end{cases}
	\end{equation*}
	
	In view of the connection between the AMO and the UAMO, we thus call $W_{\lambda_1,\lambda_2,\Phi,\theta,s}$ the \emph{mosaic unitary almost-Mathieu operator}, or the \emph{mosaic UAMO} for short.

	In contrast to the UAMO, the coin sequence for the mosaic UAMO is almost-periodic, but no longer quasi-periodic. However, we can still recover quasi-periodicity in the study of the eigenvalue equation by passing to steps of length $s$, an idea which has been fruitfully applied in several similar models, see, e.g.\ \cite{DFG2023JST, wangExactMobilityEdges2021a}.
	
	\begin{remark} Let us make a few comments.
		\begin{enumerate}[itemsep=1ex]
			\item With the single coupling constant of the AMO being replaced by two independent coupling constants for the UAMO, one might be tempted to consider another mosaic model by setting $\lambda_1=1$ at every $s$-th site. However, as noted in \cite[Remark 2.1(c)]{CFO1} the quantity that most closely parallels the coupling constant of the AMO is
			\begin{equation*}
				\lambda_0=\lambda_0(\lambda_1,\lambda_2):=\frac{\lambda_2(1+\lambda_1')}{\lambda_1(1+\lambda_2')}.
			\end{equation*}
			In view of this, the only way to make $\lambda_0$ vanish within the admitted parameter ranges $\lambda_1,\lambda_2\in[0,1]$ is to set $\lambda_2=0$, which motivates the definition of the mosaic model that we use here.
			\item Setting $s=1$ in \eqref{eq:mosaic_coins} one recovers exactly the UAMO from \cite{CFO1}.
		\end{enumerate}
	\end{remark}
	
	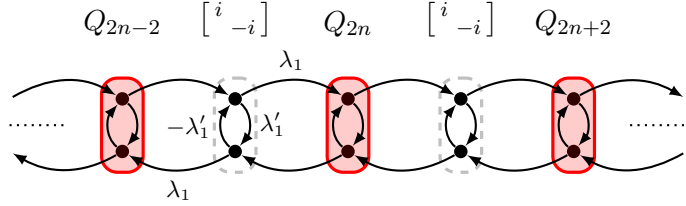
\begin{figure}[t]
		\begin{center}
			\begin{tikzpicture}[baseline={($ (current bounding box.center) + (0,.05) $)},scale=1]
				\def\xdist{1.5}
				\def\ydist{.7}
				\tikzstyle{dot} =[circle,fill,inner sep = 1.8]
				\tikzstyle{rects} = [rectangle, rounded corners=6, very thick, draw=red, fill=red, fill opacity=.2,inner sep=1.8mm, label={[above=.4cm]:#1}]
				\tikzstyle{rectstriv} = [rectangle, rounded corners=6, very thick, draw=gray!50, dashed, inner sep=1.8mm, label={[above=.4cm]:#1}]
				\tikzstyle{arr} = [->,thick,black,>=latex]
				\tikzstyle{l1} = [midway,auto,font=\footnotesize]
				\foreach \x [evaluate] in {0,...,4}{
					\foreach \y in {1,...,2}{		
						\node[dot] (\x\y) at (\xdist*\x,+\ydist-\ydist*\y) {};
					}
				}
				\node[rects={$Q_{2n-2}$},fit=(01) (02)] {};
				\node[rectstriv={$\left[\begin{smallmatrix}i&\\&-i\end{smallmatrix}\right]$},fit=(11) (12)] {};
				\node[rects={$Q_{2n}$},fit=(21) (22)] {};
				\node[rectstriv={$\left[\begin{smallmatrix}i&\\&-i\end{smallmatrix}\right]$},fit=(31) (32)] {};
				\node[rects={$Q_{2n+2}$},fit=(41) (42)] {};
				\draw[dotted,thick] (-\xdist,+.5*\ydist-\ydist) -- (-.5*\xdist,+.5*\ydist-\ydist);
				\draw[dotted,thick] (4.5*\xdist,+.5*\ydist-\ydist) -- (5*\xdist,+.5*\ydist-\ydist);
				\node[inner sep=0,minimum width=.15cm] (leftup) at (-\xdist,\ydist-1*\ydist) {};
				\node[inner sep=0,minimum width=.15cm] (rightup) at (5*\xdist,\ydist-1*\ydist) {};
				\node[inner sep=0,minimum width=.15cm] (leftdown) at (-\xdist,\ydist-2*\ydist) {};
				\node[inner sep=0,minimum width=.15cm] (rightdown) at (5*\xdist,\ydist-2*\ydist) {};
				\draw[arr] (leftup) to[out=30,in=150] (01);
				\draw[arr] (01) to[out=30,in=150] (11);
				\draw[arr] (11) to[out=30,in=150] node[l1] {$\lambda_1$} (21);
				\draw[arr] (21) to[out=30,in=150] (31);
				\draw[arr] (31) to[out=30,in=150] (41);
				\draw[arr] (41) to[out=30,in=150] (rightup);
				\draw[arr] (02) to[out=-150,in=-30] (leftdown);
				\draw[arr] (12) to[out=-150,in=-30] node[l1] {$\lambda_1$} (02);
				\draw[arr] (22) to[out=-150,in=-30] (12);
				\draw[arr] (32) to[out=-150,in=-30] (22);
				\draw[arr] (42) to[out=-150,in=-30] (32);
				\draw[arr] (rightdown) to[out=-150,in=-30] (42);
				
				\draw[arr] (12) to[out=140,in=-130] node[l1,left] {$-\lambda_1'$} (11);
				\draw[arr] (11) to[out=-40,in=50] node[l1,right] {$\lambda_1'$} (12);
				
				\foreach \x in {0,2,3,4}{
					\draw[arr] (\x2) to[out=140,in=-130] (\x1);
					\draw[arr] (\x1) to[out=-40,in=50] (\x2);
				}
			\end{tikzpicture}
		\end{center}
		\caption{The mosaic UAMO for $s=2$. The arrows indicate the action of $S_{\lambda_1}$, where for the sake of clarity the parameters are displayed only at a single lattice site. In the red cells, the respective non-trivial $Q_n$ is acting, while in the grayed out cells the trivial coin with $\lambda_2=0$ acts.}
	\end{figure}
	
	Theorem \ref{thm:mobs} thus shows that the mosaic UAMO exhibits an explicit mobility edge for suitable choices of the parameters. As said before, this gives a new type of phase transition in the world of one-dimensional quasi-periodic quantum walk operators: a phase transition in the spectral parameter.

	\section{Generalized Extended CMV Matrices}
	\label{sec:GECMV}
	
	\subsection{Gauge transformation}
	\label{sec:unitaryEquivGECMVtoECMV}
	
	We first prove that the phases of the $\rho$'s that define a GECMV matrix can be freely changed via a diagonal gauge transformation:
	
	\begin{proof}[Proof of Theorem~\ref{thm:rho_phases}.]
		First, note that to prove the statement it is sufficient to show that any GECMV matrix $\mathcal E$ with Verblunsky coefficients $\alpha_n$ can be transformed by a diagonal unitary into a ``reference'' GECMV matrix $\mathcal E_0$ with the same $\alpha$'s. This readily implies the statement by combining two such steps: if $\mathcal E^\xi$ and $\mathcal E^\chi$ are two such GECMV matrices with $D_\xi^*\mathcal E^\xi D_\xi=\mathcal E_0$ and $D_\chi^*\mathcal E^\chi D_\chi=\mathcal E_0$, respectively, then $\mathcal E^\xi=D_\xi D_\chi^*\mathcal E^\chi D_\chi D_\xi^*$. 
		A particularly natural choice for $\mathcal E_0$ turns out to be the standard extended CMV matrix with Verblunksy coefficients $\alpha_n$ and $\rho_n=\sqrt{1-|\alpha_n|^2}$.
		
		Let $\mathcal E^\xi$ be a GECMV matrix  as in \eqref{eq:gecmv} that is specified by the Verblunsky coefficients $\alpha_n$ and $\xi_n\rho_n$. We show that $\mathcal E^\xi$ is unitarily equivalent to $\mathcal E_0$ via a diagonal unitary operator. 
		Fix $d_0,d_{-1} \in \partial \bbD$ and define the entries of $D$ recursively by
		\begin{equation}\label{eq:lambdas}
			d_{2n+2}=\xi_{2n+1}^{-1}\xi_{2n}^{-1}d_{2n},\qquad d_{2n+1}=\xi_{2n-1}^{-1}\xi_{2n}^{-1}d_{2n-1}.
		\end{equation}
		We then define the new Verblunsky coefficients
		\begin{equation}\label{eq:alpha_realified}
			\tilde\alpha_{k}=\frac{d_0}{d_{-1}}\xi_{-1}\alpha_{k},\qquad\tilde\rho_{k}=\rho_{k},
		\end{equation}
		and denote by $\tilde{\mathcal E}$ the extended CMV matrix corresponding to $\tilde\alpha$ and $\tilde\rho$. To conclude, we will  demonstrate
		\begin{equation}\label{eq:ERealified}
			\tilde{\mathcal E}=D^*\mathcal E^\xi D.
		\end{equation}
		From the recursion relation \eqref{eq:lambdas} we get
		\begin{align*}
			\frac{d_{2n}}{d_{2n-1}}=\frac{\xi_{-1}}{\xi_{2n-1}}\frac{d_0}{d_{-1}}, \qquad \frac{d_{2n}}{d_{2n+1}}	=\xi_{2n}\xi_{-1}\frac{d_0}{d_{-1}}
		\end{align*}
		We then calculate that for all integers $n$,
		\begin{align}
			\overline{d_{2n}}d_{2n+2}(\xi_{2n+1}\rho_{2n+1})(\xi_{2n}\rho_{2n})=\rho_{2n+1}\rho_{2n}&=\tilde\rho_{2n+1}\tilde\rho_{2n},	\label{eq:t_rho_1}\\
			\overline{d_{2n+1}}d_{2n-1}(\overline{\xi_{2n-1}\rho_{2n-1})}(\overline{\xi_{2n}\rho_{2n}})=\overline{\rho_{2n-1}}\overline{\rho_{2n}}&=\overline{\tilde{\rho}_{2n-1}}\overline{\tilde{\rho}_{2n}} \label{eq:t_rho_2}\\
			d_{2n+2}\overline{d_{2n+1}}(\xi_{2n+1}\rho_{2n+1})\alpha_{2n}  &=\tilde\rho_{2n+1}\tilde\alpha_{2n}
			\label{eq:t_alrho_1}\\
			d_{2n}\overline{d_{2n+1}}(\overline{\xi_{2n}\rho_{2n}})\alpha_{2n-1}&=    	\overline{\tilde\rho_{2n}}\tilde\alpha_{2n-1}.	\label{eq:t_alrho_2}
		\end{align}
		This suffices to prove \eqref{eq:ERealified}.
		
		The statement of the theorem follows from \eqref{eq:alpha_realified} by noting that we may choose $d_0$ and $d_{-1}$ so that $d_0\xi_{-1}/d_{-1}=1$ which yields $\mathcal E^\xi=D^*\mathcal E_0 D$.
	\end{proof}

	\subsection{Reflection Symmetries}\label{sec:reflection}
	Consider the GECMV matrix in \eqref{eq:gecmv} with Verblunsky pairs $(\alpha_{j},\rho_{j})$. For $k \in \bbZ$, let $R_{k}$ be the unitary involution on $\ell^2(\bbZ)$ that reflects through the center $c=k+\tfrac{1}{2}$, that is, $R_k:\delta_n \mapsto \delta_{-n+2k+1}$. 
	In particular, $R_k:\delta_{2n-1} \mapsto \delta_{2(-n+k)}$ and $R_k:\delta_{2n} \mapsto \delta_{2(-n+k)+1}$.
	Notice that $R_{k}$ maps $\ell^2(\{-n,\ldots,n+2c\})$ to itself.

	\begin{definition}[Reflection]\label{def.radialReflc}
		We call $\mathcal{E}^{\reflected} :=R_{k}\mathcal{E}R_{k}$ the \emph{reflection} of $\mathcal{E}$ with center $c=k+1/2$. One can check that $\mathcal{E}^{\reflected}$ is obtained from $\mathcal E$ by exchanging the positions of the elements that are symmetric with respect to the center of the square 
		\begin{equation*}
			\left[\begin{array}{c|c} 
				-\overline{\alpha_{k}}\alpha_{k-1}&\overline{\alpha_{k+1}}\rho_{k}\\ 
				\hline 
				-\overline{\rho_{k}}\alpha_{k-1}&-\overline{\alpha_{k+1}}\alpha_{k}
			\end{array}\right]
		\end{equation*}
		when $k$ is even, and
		\begin{equation*}
			\left[\begin{array}{c|c} 
				-\overline{\alpha_{k}}\alpha_{k-1}&\alpha_{k-1}\rho_{k}\\ 
				\hline 
				-\overline{\rho_{k}}\overline{\alpha_{k+1}}&-\overline{\alpha_{k+1}}\alpha_{k}
			\end{array}\right]
		\end{equation*}
		when $k$ is odd.
	\end{definition}
	
	\begin{remark}\mbox{}
		\begin{enumerate}[itemsep=1ex]
			\item We restrict ourselves to centers from $\frac{1}{2}+\bbZ$. This is mostly for convenience so that the reflected GECMV matrix is again a GECMV matrix. If one reflects through an integer center, the reflected matrix is the \emph{transpose} of a GECMV matrix.
			\item In the quantum walks language of Section \ref{sec:MUAMO}, if $k$ is even, the center of reflection lies ``between'' the cells at $k$ and $k+1$, whereas if $k$ is odd, the center of reflection lies ``within'' the cell at $k$.
		\end{enumerate}
	\end{remark}
	
	A direct consequence of this definition is that $\mathcal E$ and $\mathcal E^\reflected$ have the same spectrum with similar statements for suitable finite cutoffs. In particular, for the finite restriction (or ``cutoff'' GECMV matrix \cite{simonCMVMatricesFive2007}) $\mathcal{E}|_{[-n,n+2c]}$ one has for the reflection $\mathcal E^\reflected$ with center $c$ that
	\begin{equation}\label{eq.determinant}
		\det(z\idty-\mathcal{E}|_{[-n,n+2c]}^{\reflected})=\det(z\idty-\mathcal{E}|_{[-n,n+2c]}).
	\end{equation}
	
	Let us see how one can take advantage of some of these ideas in the setting of GECMV matrices generated by sampling functions with suitable symmetries. Concretely, assume that $\{\mathcal{E}(\theta)\}_{\theta \in \bbT} = \{\mathcal{E}(\alpha(\theta), \rho(\theta)\}_{\theta \in \bbT}$ is a family of GECMV matrices that depends on the variable $\theta \in \bbT$, and let us furthermore assume that the coefficients possess the following reflection property with respect to the reflection center $c$:
	\begin{equation}\label{eq.newReflectionProp}
		\alpha_{-n+2c}(\theta)=\overline{\alpha_{n+2c}(-\theta)},\quad \rho_{-n+2c}(\theta)=-\overline{\rho_{n+2c}(-\theta)}.
	\end{equation}
	Then, the corresponding GECMV matrix $\mathcal{E}(\theta)$ satisfies $\mathcal{E}^{\reflected}(\theta)=S^{-2(k+1)}\mathcal{E}(-\theta)S^{2(k+1)}$ where $S:\delta_n\mapsto\delta_{n+1}$ denotes the bilateral shift on $\ell^2(\bbZ)$. That is, reflecting with center $c=k+1/2$ is equivalent to shifting by $2(k+1)$ up to a sign-change of $\theta$.
	
	This yields the following result:
	\begin{prop}\label{prop:deterEveness1}
		Let $\mathcal E(\theta)$ be a GECMV matrix with Verblunsky coefficients satisfying \eqref{eq.newReflectionProp} for $c=-1/2$. Then  $\det(z\idty-\mathcal{E}(\theta)|_{[-n,n-1]})$ is an even function of $\theta\in\bbT$.
	\end{prop}
	\begin{proof}
		As a consequence of the reflection pairs above, \eqref{eq.determinant} and \eqref{eq.newReflectionProp} we have 
		\begin{equation}\label{eq.deterEveness1}
			\det(z\idty-\mathcal{E}(\theta)|_{[-n,n-1]})=\det(z\idty-\mathcal{E}^{\reflected}(\theta)|_{[-n,n-1]})(\theta)=\det(z\idty-\mathcal{E}(-\theta)|_{[-n,n-1]}).
		\end{equation}
	\end{proof}
	
	We shall apply this result in Section \ref{sec:localization} to prove localization of the mosaic UAMO in the supercritical regime. 
	We remark that when applied to the UAMO from \cite{CFO1}, Proposition \ref{prop:deterEveness1} provides an alternate proof for \cite[Lemma 4.2]{yangLocalizationMagneticQuantum2022}.

	\section{Cocycle Dynamics and Lyapunov Exponents} \label{sec:lyap}
	
	A crucial ingredient in the study of the properties of a GECMV matrix is the classification of cocycle behavior via Avila's global theory of one-frequency analytic cocycles \cite{Avila2015Acta}.
	We first review this theory and then show the equivalence between transfer matrix cocycles as defined in \eqref{eq:GECMV_transmat} and Szeg\H{o} cocycles. This will provide the necessary tools to calculate the Lyapunov exponent on the spectrum.
	
	\subsection{Review of Avila's global theory}
	
	Given $\Phi$ irrational and $M:\bbT\to\bbM(2,\bbC)$ continuous, consider the skew product
	\begin{equation}\label{eq:qp_cocycle}
		(\Phi,M):\bbT\times\bbC^2\to\bbT\times\bbC^2,\qquad (\theta,v)\mapsto(\theta+\Phi,M(\theta)v).
	\end{equation}
	The iterates of this \emph{quasi-periodic cocycle} are given by $(\Phi,M)^n=(n\Phi,M^n)$ where for $n\in\bbN$
	\begin{equation*}
		M^n(\theta)=M^{n,\Phi}(\theta)=\prod_{j=n-1}^0M(\theta+j\Phi).
	\end{equation*}
	The \emph{Lyapunov exponent} of the cocycle $(\Phi,M)$ is defined by
	\begin{equation*}
		L(\Phi,M)=\lim_{n\to\infty}\frac1n\int_\bbT\log\|M^{n,\Phi}(\theta)\| d\theta.
	\end{equation*}
	If $M$ is analytic with an analytic extension to a strip $\bbT_\delta := \{\theta+i\epsilon:|\epsilon|<\delta\}$, for $|\epsilon|<\delta$ we may consider the complexified cocycle $M(\cdot+i\epsilon):\theta\mapsto M(\theta+i\epsilon)$ and define $L(\Phi, M,\epsilon)$ as the Lyapunov exponent associated with the complexified cocycle map $M(\cdot + i\epsilon)$, that is,
	\begin{equation}
		L(\Phi, M, \epsilon)=L(\Phi, M(\cdot+i\epsilon)).
	\end{equation}
	Under the analyticity assumption, we define the \emph{acceleration} \cite{Avila2015Acta,JitoMarx2012CMP,JitoMarx2012CMPErr} for $|\eta|<\delta$ by 
	\begin{equation*}
		\omega(\Phi, M, \eta):=\lim_{\epsilon\downarrow0}\frac1{2\pi\epsilon}(L(\Phi,M,\eta+\epsilon)-L(\Phi,M,\eta)),
	\end{equation*}
	and abbreviate
	\begin{equation*}
		\omega(\Phi, M):= \omega(\Phi,M,0) = \lim_{\epsilon\downarrow0}\frac1{2\pi\epsilon}(L(\Phi,M,\epsilon)-L(\Phi,M)).
	\end{equation*}
	
	A central property of the acceleration that we shall need further below is its quantization, that is, 
	\begin{equation*}
		\omega(\Phi, M, \eta)\in\tfrac12\bbZ
	\end{equation*}
	for all $|\eta|<\delta$ \cite{Avila2015Acta, JitoMarx2012CMP, JitoMarx2012CMPErr}. Moreover, if $M(\theta)\in\SL(2,\bbC)$ for all $\theta \in \bbT$, we have $\omega(\Phi, M, \eta)\in\bbZ$ for all $|\eta|<\delta$.
	
	A $\SL(2,\bbC)$-cocycle $(\Phi,M)$ is called \emph{uniformly hyperbolic} if for some constants $c,\lambda>0$ one has
	\begin{equation}
		\|M^n(\theta)\| \geq ce^{\lambda | n |}
	\end{equation}
	uniformly in $n \in \bbZ$ and $\theta \in \bbT$. From the spectral perspective, uniform hyperbolicity corresponds to the resolvent set of the underlying operator in the sense that a given spectral parameter $z$ belongs to the resolvent set if and only if the associated transfer matrix cocycle is uniformly hyperbolic \cite{DFLY2016DCDS}; see also \cite{Johnson1986JDE, Zhang2020JST}. 
	
	\begin{definition}
		Assume that $(\Phi,M)$ is a $\SU(1,1)$ cocycle that is not uniformly hyperbolic. Then  $(\Phi,M)$ is said to be
		\begin{enumerate}[itemsep=1ex]
			\item \emph{Supercritical}, if $L(\Phi,M)>0$.
			\item \emph{Subcritical}, if there exists $\epsilon_0>0$ such that $L(\Phi, M,\epsilon)=0$ for all $\epsilon$ with $|\epsilon|<\epsilon_{0}.$
			\item \emph{Critical}; otherwise.
		\end{enumerate}
	\end{definition}
	
	\subsection{Calculations of Lyapunov exponent}
	In this section, we compute the Lyapunov exponent of the mosaic UAMO model. Let us first introduce the basic notations and definitions:
	For the mosaic UAMO with $s=2$ and  Verblunsky coefficients given in \eqref{eq:mosaicVerblunskies}, the transfer matrices from \eqref{eq:GECMV_transmat} take the form:
	\begin{equation}
		\begin{aligned}\label{evenA}
			A_{2n,z}=&\frac{1}{\lambda_{2}\cos2\pi(\theta+2n\Phi)-i\lambda_{2}'}\\
			&\times\begin{bmatrix}
				\lambda_{1}^{-1}z^{-1}+2\lambda_{1}^{-1}\lambda_{1}'\lambda_{2}\sin2\pi(\theta+2 n\Phi)+\lambda_{1}^{-1}{\lambda_{1}'}^{2}z&-\lambda_{2}\sin2\pi(\theta+2 n\Phi)-\lambda_{1}'z\\
				-\lambda_{2}\sin2\pi(\theta+2 n\Phi)-\lambda_{1}'z&\lambda_{1}z
			\end{bmatrix},
		\end{aligned}
	\end{equation}
	and
	\begin{equation}\label{oddA}
		\begin{aligned}
			A_{2n+1,z}&=
			i\begin{bmatrix}
				\lambda_{1}^{-1}z^{-1}+\lambda_{1}^{-1}{\lambda_{1}'}^{2}z&-\lambda_{1}'z\\
				-\lambda_{1}'z&\lambda_{1}z
			\end{bmatrix}.
		\end{aligned}
	\end{equation}
	These transfer matrices naturally define a \emph{transfer matrix cocycle} of the form above: define $A_{\lambda_{1},\lambda_{2},z}: \bbT \to \bbC^{2\times 2}$ by 
	\begin{equation}\label{eq.transferMatrix}
		A_{\lambda_{1},\lambda_{2},z}(\theta)=\frac{1}{\lambda_{2}\costwopi(\theta)-i\lambda_{2}'}
		\begin{bmatrix}
			\lambda_{1}^{-1}z^{-1}+2\lambda_{1}^{-1}\lambda_{1}'\lambda_{2}\sintwopi(\theta)+\lambda_{1}^{-1}{\lambda_{1}'}^{2}z&-\lambda_{2}\sintwopi(\theta)-\lambda_{1}'z\\
			-\lambda_{2}\sintwopi(\theta)-\lambda_{1}'z&\lambda_{1}z
		\end{bmatrix},
	\end{equation}
	where we adopted the notation
	\begin{equation}\label{eq:cstwopi}
		\costwopi(\theta)=\cos(2\pi\theta),\qquad\sintwopi(\theta)=\sin(2\pi\theta).
	\end{equation}
	With this definition, one readily checks that 
	\begin{equation}
		A_{2n,z} = A_{\lambda_{1},\lambda_{2},z}(\theta+2n\Phi),
		\qquad A_{2n+1,z}=A_{\lambda_{1},0,z}(\theta + 2n\Phi).
	\end{equation}
	We remark that this construction generalizes in a straightforward fashion to $s>2$. 
	
	In order to formulate results for a genuine quasi-periodic cocycle instead of the merely almost-periodic $A_z$, let us define the two-step cocycle map by
	\begin{equation} \label{eq:twoblockAzdef}
		A^{+}_z(\theta) \equiv  A^{+}_{\lambda_1,\lambda_2,z}(\theta) 
		:= A_{\lambda_1,0,z}(\theta) A_{\lambda_1,\lambda_2,z}(\theta). 
	\end{equation}
	From the definitions, $A^+_z$ establishes a quasi-periodic cocycle in the sense of \eqref{eq:qp_cocycle}, i.e.,
	\begin{equation*}
		(2\Phi,A_z^+): \bbT \times \bbC^2 \to \bbT \times \bbC^2, \quad (x,v)\mapsto (x+2\Phi, A_z^+(x) v).
	\end{equation*}
	From the definitions above, the reader can confirm that its iterates are given by
	\begin{equation} \label{eq:twoblockAzprod}
		A_{2n-1,z} \cdots A_{1,z}A_{0,z} = \prod_{j=n-1}^0 A^{+}_z(\theta+ 2j\Phi).
	\end{equation}
	
	Consequentially, the \emph{Lyapunov exponent} associated to the mosaic UAMO is defined to be half of the Lyapunov exponent of the quasi-periodic cocycle $(2\Phi,A^{+}_z)$, that is,
	\begin{equation} \label{eq:mosaicLEDef}
		L(z) 
		= \frac{1}{2} L(2\Phi,A^{+}_z) 
		=  \lim_{n\to\infty}\frac{1}{2n} \int_{\bbT} \log\|A^{+}_z(\theta+2(n-1)\Phi) \cdots A^{+}_z(\theta)\| \, d\theta.
	\end{equation}

	\begin{theorem}\label{Lyapunov}
		For $s=2$, any $\lambda_1,\lambda_2 \in (0,1)$, $\Phi \in \bbR\backslash \bbQ$, and $e^{it}\in \Sigma_{\lambda_{1},\lambda_{2},\Phi,2}$, the Lyapunov exponent  of the associated mosaic UAMO model is given by 
		\begin{equation}\label{lya}
			L(e^{it}) = \frac{1}{2} \max\left\{0,F(\lambda_{1},\lambda_{2},t)\right\},
		\end{equation}
		where we denote 
		\begin{equation*}
			F(\lambda_{1},\lambda_{2},t)=\log\left[\frac{\lambda_{2}}{\lambda_{1}^{2}(1+\lambda_{2}')}\left(2\lambda_{1}'|\cos t|+\sqrt{\lambda_{1}^{4}+4{\lambda_{1}'}^{2}\cos^{2}t}\right)\right].
		\end{equation*}
		Moreover, for any $e^{it}\in \Sigma_{\lambda_{1},\lambda_{2},\Phi ,2}$, the cocycle $(2\Phi,A^{+}_{e^{it}})$ is 
		\begin{enumerate}[label={\rm (\alph*)}, itemsep=1ex]
			\item \emph{subcritical}  if and only if $F(\lambda_{1},\lambda_{2},t)<0$.
			\item \emph{critical} if and only if $F(\lambda_{1},\lambda_{2},t)=0$.
			\item \emph{supercritical} if and only if $F(\lambda_{1},\lambda_{2},t)>0$.
		\end{enumerate}
	\end{theorem}

	According to Corollary \ref{cor:CMV_realified}, it suffices to  compute the Lyapunov exponents for the corresponding extended CMV matrix. Although the calculations can be done with the initial cocycle maps $A^{+}_{z}$, it is more convenient to put the question into $\SU(1,1)$ which allows one to directly apply Avila's global theory:
	\begin{lemma}\label{lem:CocycleEquivalence}
		Given $(\alpha,\rho)$ and $z\in\partial\bbD$, let $A_{n,z}$ be the transfer matrix cocycle given by \eqref{eq:GECMV_transmat} with $\rho_n$ replaced by $|\rho_n|$. Then we have the following:
		\begin{equation}\label{eq.ConjugatedTransferMatrix}
			A_{n,z}=R_{2n}^{-1}JS^{+}_{n,z}JR_{2n-2},
		\end{equation}
		where $S^{+}_{n,z}=S_{2n,z}S_{2n-1,z}$ is determined by the \emph{normalized Szeg\H{o} cocycle maps}
		\begin{equation}\label{eq:szego_normalized}
			S_{n,z}=\frac{z^{-\frac{1}{2}}}{|\rho_{n}|}\begin{bmatrix}z&-\overline{\alpha_{n}}\\-\alpha_{n}z&1\end{bmatrix} \in \SU(1,1),
		\end{equation}
		and
		\begin{equation}
			R_{n}=\begin{bmatrix}1&0\\-\overline{\alpha_{n}}& |\rho_{n}| \end{bmatrix},\qquad J=\begin{bmatrix}0&1\\1&0\end{bmatrix}.
		\end{equation}
	\end{lemma}
	\begin{proof}
		This follows from a direct computation.
	\end{proof}

	Due to \eqref{eq.ConjugatedTransferMatrix}, it suffices to consider the four-step combined quasi-periodic cocycle $(2\Phi,S^{++}_{z})$ instead of $(2\Phi,A^{+}_{z})$, where
	\begin{equation}\label{eq.combinedSzego}
		S^{++}_{n,z}
		=S^{++}_{z}(\theta+2n\Phi)=S^{+}_{2n+1,z}S^{+}_{2n,z}.
	\end{equation} 
	By direct calculations, one verifies that
	\begin{align}\label{eq.zgo4steps}
		S^{++}_{z}(\theta)&=\frac{\lambda_{1}^{-2}}{|\lambda_{2}\costwopi(\theta)-i\lambda_{2}'|}\begin{bmatrix}
			{\lambda_{1}'}^{2}+z^{2}+\lambda_{1}'\lambda_{2}(z+z^{-1})\sintwopi(\theta)&-\lambda_{1}'(1+z^{-2})-\lambda_{2}\sintwopi(\theta) (z+{\lambda_{1}'}^{2}z^{-1})\\
			-\lambda_{1}'(1+z^{2})-\lambda_{2}\sintwopi(\theta) (z^{-1}+{\lambda_{1}'}^{2}z)&{\lambda_{1}'}^{2}+z^{-2}+\lambda_{1}'\lambda_{2}(z+z^{-1})\sintwopi(\theta)
		\end{bmatrix}\\
		&=:\frac{\lambda_{1}^{-2}}{|\lambda_{2}\costwopi(\theta)-i\lambda_{2}'|}\;M_{z}(\theta).\label{eq:defMz}
	\end{align}
	
	We denote 
	\begin{align*}
		w(\theta) =
		|\lambda_2 \costwopi(\theta) - i\lambda_2'|
		= \sqrt{\lambda_2^2 \costwopi^2(\theta) + 1-\lambda_2^2} 
		= \sqrt{1-\lambda_2^2 \sintwopi^2(\theta)}.
	\end{align*}
	Note that the analytic extension of $S^{++}_{z}(\theta)$ is $M_{z}(\theta+i\epsilon)/(\lambda_{1}^{2}w(\theta+i\epsilon))$ and not $M_{z}(\theta+i\epsilon)/(\lambda_{1}^{2}(\lambda_{2}\costwopi(\theta+i\epsilon)-i\lambda_{2}'))$. This could affect calculations of the matrix norm, since, for $\epsilon \neq 0$, one can check that $|w(\theta+i\epsilon)|$ and $|\lambda_2 \costwopi(\theta+i\epsilon) - i \lambda_2'|$ need not coincide. To calculate the Lyapunov exponent of the cocycle $(2\Phi,S^{++}_{z})$, we first deal with the normalizing factor in front. By inspection, $w(\theta)$ is real-analytic on $\bbT$, and has an analytic extension to the strip $|\epsilon| < \frac{1}{2\pi} \arcsinh\sqrt{\lambda_2^{-2}-1}$ given by
	\begin{equation*}
		{w}(\theta+i\epsilon) = \sqrt{1-\lambda_2^2 \sintwopi^2(\theta + i\epsilon)}.
	\end{equation*}
	Thus, this is the expression whose integral one needs to calculate: 
	
	\begin{lemma}\label{jenson}\label{lemma51}
		Given $0 \le t \leq 1$, denote $t' = \sqrt{1-t^2}$ and $\epsilon_0 = \epsilon_0(t)= \tfrac{1}{2\pi} \arcsinh(t'/t)$. Then
		\begin{align}
			\int_0^1 \log\left| \sqrt{1-t^2 \sintwopi^2(\theta + i\epsilon)} \right| \, d\theta 
			&= \begin{dcases}
				\log\left[\frac{1+t'}{2} \right]-2\pi(\epsilon+\epsilon_0) &\epsilon\leq-\epsilon_0,\\
				\log\left[\frac{1+t'}{2} \right] & -\epsilon_0\leq\epsilon\leq\epsilon_0,\\
				\log\left[\frac{1+t'}{2} \right]+2\pi(\epsilon-\epsilon_0) &\epsilon\geq\epsilon_0,
			\end{dcases}\\
			&=\log\left[\frac{1+t'}{2} \right]+2\pi\max\{0,|\epsilon|-\epsilon_0\}.
		\end{align}
	\end{lemma}
	
	\begin{proof}
		Note that
		\begin{align*}
			\log\left| \sqrt{1-t^2 \sintwopi^2(\theta + i\epsilon)} \right|
			& = \frac{1}{2} \log|g_\epsilon(e^{2\pi i \theta})|,
		\end{align*}
		where 
		\begin{equation*}
			g_\epsilon(z) = z^2 + \tfrac{t^2}{4} \left( z^4 e^{-4\pi\epsilon }  - 2z^2 + e^{4\pi\epsilon}\right).
		\end{equation*}    
		Solving $g(z) = 0$ gives the four roots
		\begin{align*}  
			\pm \sqrt{
				\frac{\frac{t^2}{2}-1 \pm \sqrt{(1-\frac{t^2}{2})^2 - \frac{1}{4}t^4}}{\frac{1}{2} t^2 e^{-4\pi \epsilon}} 
			}
			& = \pm \sqrt{
				\frac{\frac{t^2}{2}-1 \pm \sqrt{1-t^2}}{\frac{1}{2} t^2 e^{-4\pi \epsilon}} 
			} = \pm i
			\frac{1 \mp t'}{ t e^{-2\pi \epsilon} 
			}.
		\end{align*}
		For $|\epsilon|<\epsilon_0$, the only roots of $g$ in $\bbD$ are
		\begin{equation*}
			r_\pm = \pm i
			\frac{1 - t'}{ t e^{-2\pi \epsilon} 
			}.
		\end{equation*}
		Applying Jensen's formula to $g$, we obtain
		\begin{align*}
			\frac{1}{2} \int_0^1 \log|g_\epsilon(e^{2\pi i \theta})| \, d\theta 
			= \frac{1}{2} \left(\log|g_\epsilon(0)| - \log|r_{+}| - \log|r_{-}|\right) 
			= \log\left(\frac t2\right) + 2\pi \epsilon -  \log\left| \frac{1-t'}{te^{-2\pi \epsilon}} \right|
		\end{align*}
		which yields the desired result.
		The case $|\epsilon|>\epsilon_0$ is similar.
	\end{proof}
	
	Define 
	\begin{equation}\label{eq.scalarFactorCont}
		\tilde{\gamma}=\log\left[\frac{\lambda_{1}^{2}(1+\lambda_{2}')}{2}\right].
	\end{equation}
	It follows immediately from the above lemma that the contribution of the scalar factors of the cocycle maps to the Lyapunov exponent is given by the following quantity:
	\begin{equation}\label{eq:lyap_scalar_contribution}
		-\frac{1}{2}\int_{\bbT}\log|\rho_{3}(\theta)\rho_{2}\rho_{1}\rho_{0}|d\theta=-\frac{1}{2}\log\left[\frac{\lambda_{1}^{2}(1+\lambda_{2}')}{2}\right]=-\frac12\tilde\gamma.
	\end{equation}
	
	\begin{proof}[Proof of Theorem \ref{Lyapunov}]
		Denote $z=e^{it}$,
		in view of Lemma \ref{lem:CocycleEquivalence}, it suffices to show that
		\begin{equation}\label{lya1}
			L(2\Phi,S^{++}_{e^{it}}) = \max\left\{0,F(\lambda_{1},\lambda_{2},t)\right\}.
		\end{equation}
		We first complexify the phase by letting $\theta\mapsto\theta+i\epsilon$. Then, by the definition of Lyapunov exponent and \eqref{eq:defMz},
		\begin{equation}\label{rela}
			L(2\Phi,S^{++}_{e^{it}}(\cdot+i\epsilon))=L(2\Phi,M_{e^{it}}(\cdot+i\epsilon))-\int_{\bbT}\log\lambda_{1}^{2}|w(\theta+i\epsilon)| \, d\theta.
		\end{equation}
		
		From this and Lemma \ref{jenson}, it is easy to check that $(2\Phi,S^{++}_{e^{it}}(\cdot+i\epsilon))$  admits a holomorphic extension to the strip $ |\epsilon|<\epsilon_{0}=\frac{1}{2\pi} \arcsinh(\lambda_2^{-1}\lambda_2')$.
		We conclude that $(2\Phi,S^{++}_{e^{it}}(\cdot))$ and $(2\Phi,M_{e^{it}}(\cdot))$ have the same acceleration whenever $ |\epsilon|<\epsilon_{0}$, that is, 
		\begin{equation}\label{rela1}
			\omega(2\Phi,S^{++}_{e^{it}}(\cdot+i\epsilon))= \omega(2\Phi,M_{e^{it}}(\cdot+i\epsilon)) , \qquad \forall  |\epsilon|<\epsilon_{0}.
		\end{equation}
		
		Now let us calculate the Lyapunov exponent of  $(2\Phi,M_{e^{it}}(\cdot+i\epsilon))$ as $\epsilon \to \infty$. For large $\epsilon>0$, we have by the definition of $M_z$ in \eqref{eq:defMz}
		\begin{align*}
			M_{e^{it}}(\theta+i\epsilon)&=e^{2\pi\epsilon}\left(\begin{bmatrix}
				e^{-2\pi\epsilon}\lambda_{1}'\lambda_{2}(z+z^{-1})\sintwopi(\theta+i\epsilon)&-e^{-2\pi\epsilon} (z+{\lambda_{1}'}^{2}z^{-1})\lambda_{2}\sintwopi(\theta+i\epsilon)\\
				-e^{-2\pi\epsilon}(z^{-1}+{\lambda_{1}'}^{2}z)\lambda_{2}\sintwopi(\theta+i\epsilon)&e^{-2\pi\epsilon}\lambda_{1}'\lambda_{2}(z+z^{-1})\sintwopi(\theta+i\epsilon)
			\end{bmatrix}+o(1)\right)\\
			&=e^{2\pi\epsilon}\lambda_{2}ie^{-2\pi i\theta}\left(\frac12\begin{bmatrix}
				\lambda_{1}'(z+z^{-1})&-(z^{2}+{\lambda_{1}'}^{2})z^{-1}\\
				-(z^{-2}+{\lambda_{1}'}^{2})z&\lambda_{1}'(z+z^{-1})
			\end{bmatrix}+o(1)\right).
		\end{align*}
		By continuity of Lyapunov exponent \cite{BJ}, 
		\begin{equation*}
			L(2\Phi,M_{e^{it}}(\cdot+i\epsilon))=\log\left[\frac{\lambda_{2}}2\left(2\lambda_{1}'|\cos t|+\sqrt{\lambda_{1}^{4}+4{\lambda_{1}'}^{2}\cos^{2}t}\right)\right]+2\pi\epsilon+ o(1),
		\end{equation*}
		and by quantization of acceleration  \cite{Avila2015Acta}, 
		\begin{equation}\label{eq.analyticMatrixExp}
			L(2\Phi,M_{e^{it}}(\cdot+i\epsilon))=\log\left[\frac{\lambda_{2}}2\left(2\lambda_{1}'|\cos t|+\sqrt{\lambda_{1}^{4}+4{\lambda_{1}'}^{2}\cos^{2}t}\right)\right]+2\pi\epsilon
		\end{equation}
		for $\epsilon>0$ large enough. The case $\epsilon<0$ can be dealt with in a similar fashion.
		
		On the other hand,
		by convexity, $\omega( 2 \Phi,M_{e^{it}}(\cdot+i\epsilon)) \leq 1$  for any $\epsilon\in \bbR$. Since $M_{e^{it}}(\cdot)\notin \SL(2,\bbC)$, one cannot conclude $\omega( 2 \Phi,M_{e^{it}}(\cdot+i\epsilon)) \in \bbZ$ directly. Nevertheless, 
		since 
		\begin{equation*}
			S^{++}_{e^{it}}(\cdot)\in C^{\omega}(\bbT, \SU(1,1)),
		\end{equation*}
		one may conclude that  $\omega(2 \Phi,S^{++}_{e^{it}}(\cdot+i\epsilon))\in \bbZ$ for  $|\epsilon|<\epsilon_{0}$ by Avila's global theory \cite{Avila2015Acta}.  We distinguish two cases: \\

		\textbf{Case 1: $(2\Phi,S^{++}_{e^{it}})$ is subcritical.} Assume 
		$(2\Phi,S^{++}_{e^{it}})$ is subcritical in the regime $|\epsilon|<\delta_0\leq \epsilon_{0}$. Let us note in passing that it is unknown whether the subcritical radius $\delta_0$ is exactly $\epsilon_{0}$. From the choice of $\delta_0$, we have
		\begin{equation}\label{sub}
			L(2\Phi,S^{++}_{e^{it}}(\cdot+i\epsilon))=0 \quad \forall |\epsilon|<\delta_0. 
		\end{equation}

		\textbf{Case 2: $(2\Phi,S^{++}_{e^{it}})$ is supercritical or critical.} From \eqref{rela1} and the convexity of  $L(2\Phi,M_{e^{it}}(\cdot+i\epsilon))$ it follows that $\omega(2\Phi,S^{++}_{e^{it}}(\cdot+i\epsilon))=1$ for $|\epsilon|<\epsilon_{0}$, and  $\omega(2\Phi,M_{e^{it}}(\cdot+i\epsilon)) =1$ for all $\epsilon\in\bbR$. This implies that 
		\begin{equation*}
			L(2\Phi,M_{e^{it}}(\cdot+i\epsilon))=\log\left[\frac{\lambda_{2}}2\left(2\lambda_{1}'|\cos t|+\sqrt{\lambda_{1}^{4}+4{\lambda_{1}'}^{2}\cos^{2}t}\right)\right]+2\pi\epsilon 
		\end{equation*}
		for all $\epsilon\in\bbR$, where the case $\epsilon\leq 0$ follows by real-symmetry. 
		As a consequence, by \eqref{eq:lyap_scalar_contribution} and \eqref{rela}, we have
		\begin{equation}\label{super}L(2\Phi,S^{++}_{e^{it}}(\cdot+i\epsilon)) =F(\lambda_{1},\lambda_{2},t)+2\pi\epsilon.\end{equation}
		Then \eqref{lya1} follows from \eqref{super} and \eqref{sub}.

		By Corollary~\ref{cor:CMV_realified} and \cite{DFLY2016DCDS},  $e^{it}\notin \Sigma_{\lambda_{1},\lambda_{2},\Phi ,2}$ if and only if $(2\Phi,S^{++}_{e^{it}})$  is uniformly hyperbolic. Consequently, by Avila’s global theory \cite{Avila2015Acta}, for any $e^{it}\in \Sigma_{\lambda_{1},\lambda_{2},\Phi ,2}$,  the corresponding cocycle $(2\Phi,S^{++}_{e^{it}})$  is either supercritical, critical, or subcritical. We thus only need to 
		locate the spectral parameter $e^{it}$ which is supercritical or critical.  Then $(b)$ and  $(c)$ follows immediately from \eqref{super}, and
		$(a)$ follows from $(b)$ and  $(c)$, finally \eqref{lya1} follows from  \eqref{super} and \eqref{sub}.
	\end{proof}
	
	If $\frac{\lambda_{1}^{2}\lambda_{2}'}{2\lambda_{1}'\lambda_{2}}\in(0,1)$ for given coupling constants $\lambda_{1},\lambda_{2}$ we define 
	\begin{equation}\label{eq:t_0}
		t_{0}=\arccos\left(\frac{\lambda_{1}^{2}\lambda_{2}'}{2\lambda_{1}'\lambda_{2}}\right).
	\end{equation}
	A direct consequence of Theorem \ref{Lyapunov} is the following
	\begin{coro}\label{mobilityEdge}
		Suppose that $e^{it}\in\Sigma_{\lambda_{1},\lambda_{2},\Phi ,2}$. Then 
		\begin{itemize}
			\item $L(e^{it})>0$ for $t\in[0,t_{0})\cup(\pi-t_{0},\pi+t_{0})\cup(2\pi-t_{0},2\pi]$, and
			\item $L(e^{it})=0$ for $t\in[t_{0},\pi-t_{0}]\cup[\pi+t_{0},2\pi-t_{0}]$.
		\end{itemize}
	\end{coro}

	\section{Localization in the Supercritical Regime}\label{sec:localization}
	In this section, we prove that for {\it Diophantine} frequency $\Phi$ and {\it non-resonant} phase $\theta\in\bbT$, the generalized eigenfunctions of the mosaic UAMO decay exponentially for $L(z)>0$.
	This implies Anderson localization by a standard argument and thus proves Theorem \ref{thm:mobs:pp}.
	As discussed above, the mosaic UAMO can be transformed into a standard extended CMV matrix by a suitable gauge, so our main result also gives an interesting example (Anderson localization for fixed frequency) in the theory of OPUC. 
	Moreover, we calculate the exact decay rate of the eigenfunctions.
	
	\subsection{Localization}
	
	We would like to utilize the evenness of the characteristic polynomial of the mosaic UAMO as a function of the phase $\theta$. However, inspecting Proposition \ref{prop:deterEveness1} we find that the Verblunsky coefficients of the mosaic UAMO given in \eqref{eq:mosaicVerblunskies} do not possess the required symmetry property \eqref{eq.newReflectionProp}. 
	Since we want to nevertheless utilize Proposition \ref{prop:deterEveness1} we again leverage the gauge transformation in Theorem~\ref{thm:rho_phases}, yet, this time in the reverse direction: it turns out that by rotating each $\rho$ with even index by $\frac{\pi}{2}$ anti-clockwise reveals the hidden symmetry. As a consequence, we establish the evenness of suitable characteristic polynomials, a key ingredient in the proof of localization in \cite{Jitomirskaya1999Annals}, in our proof of localization for (GE)CMV matrices.
	
	Rotating the $\rho$'s with even index as prescribed above, we introduce the following ``complexified twin'' of the mosaic UAMO model:
	\begin{equation}\label{eq.blockDef3}
		\begin{alignedat}{4}
			\alpha_{4n-1}&=\lambda_{2}\sin2\pi(\theta+2n\Phi),&\qquad\alpha_{4n+1}&=0,&\qquad\alpha_{4n}&=\alpha_{4n+2}=\lambda_{1}',\\
			\rho_{4n-1}&=\lambda_{2}\cos2\pi(\theta+2n\Phi)-i\lambda_{2}',&\qquad\rho_{4n+1}&=-i,&\qquad\rho_{4n}&=\rho_{4n+2}=i\lambda_{1}.
		\end{alignedat}
	\end{equation}
	We denote the corresponding GECMV matrix by $\mathcal{E}^{i}$ and similarly its building blocks by $\mathcal{L}^{i}$ and $\mathcal{M}^{i}$ such that $\mathcal E^i=\mathcal L^i\mathcal M^i$. 
	Comparing \eqref{eq:mosaicVerblunskies} with \eqref{eq.blockDef3}, we emphasize that the tiny change $\rho_{2n}\mapsto i\rho_{2n}$ paves the way for applying the reflection symmetry argument introduced in Section \ref{sec:reflection}:
	One easily verifies that {after the coordinate shift $\theta\mapsto\theta+\frac14$} the coefficients \eqref{eq.blockDef3} satisfy \eqref{eq.newReflectionProp} for $c=-1/2$, since the only non-constant terms have index $4n-1$ for which indeed
	\begin{equation}\label{eq:alpha_i_sym}
		\alpha_{4n-1}(\theta+\tfrac14)=\lambda_{2}\cos 2\pi(2n\Phi+\theta)=\lambda_{2}\cos 2\pi(-2n\Phi-\theta)=\alpha_{-4n-1}(-\theta+\tfrac14),
	\end{equation}
	and
	\begin{equation}\label{eq:rho_i_sym}
		\rho_{4n-1}(\theta+\tfrac14)=\lambda_{2}\sin2\pi(2n\Phi+\theta)+i\lambda_{2}'=-\lambda_{2}\sin2\pi(-2n\Phi-\theta)+i\lambda_{2}'=-\overline{\rho_{-4n-1}(-\theta+\tfrac14)}.
	\end{equation}
	
	We will also need the associated ``standard" extended CMV matrix  $\tilde{\mathcal{E}}=\tilde{\mathcal{E}}(\alpha,|\rho|)$ with every complex $\rho$ replaced by its absolute value, and we shall write $\tilde{\mathcal{E}}=\tilde{\mathcal{L}}\tilde{\mathcal{M}}$. 
	The role of $\tilde{\mathcal{E}}$ is to connect existing theory for extended CMV matrices to our setting.
	The following observation is an elementary consequence of Theorem \ref{thm:rho_phases}:
	\begin{prop}\label{prop:basicRelations}
		Let $\mathcal{E}, \mathcal{E}^{i}$ be the GECMV matrices with coefficients \eqref{eq:mosaicVerblunskies} and \eqref{eq.blockDef3}, respectively, and let $\tilde{\mathcal{E}}$ be the associated extended CMV matrix. Then 
		\begin{enumerate}[label={\rm (\arabic*)}, itemsep=1ex]
			\item $\mathcal{E}$, $\mathcal{E}^{i}$ and $\tilde{\mathcal{E}}$ are mutually unitarily equivalent.
			\item The Lyapunov exponents of the cocycles corresponding to $\mathcal{E}$, $\mathcal{E}^{i}$ and $\tilde{\mathcal{E}}$ are identical.
			\item The spectra and spectral measures of $\mathcal{E}$, $\mathcal{E}^{i}$ and $\tilde{\mathcal{E}}$ are identical.
			\item The dynamics of the solutions to the eigenvalue equations of $\mathcal{E}$, $\mathcal{E}^{i}$ and $\tilde{\mathcal{E}}$ are identical.
		\end{enumerate}
	\end{prop}
	
	Recall the definition of Diophantine numbers in \eqref{eq:diophsigmagamma} and \eqref{eq:dioph}.
	\begin{definition}[$\Phi$-resonant]
		Given $\Phi\in \DC(\kappa,\tau)$, $\kappa>0$, $\tau>1$,~$\theta\in\bbT$ is called \emph{resonant} with respect to $\Phi$ if 
		\begin{equation*}
			\left|\sin2\pi\left( \theta+ n\Phi \right)\right|<\exp(-|n|^{\frac{1}{2\tau}})
		\end{equation*}
		holds for infinitely many $n\in\bbZ$. Otherwise, $\theta$ is called \emph{non-resonant} with respect to $\Phi$.
	\end{definition}
	It is known that the collection of all Diophantine frequencies has full Lebesgue measure in $\bbT$, and the set of $\Phi$-resonant phases is a dense $G_{\delta}$-subset with zero Lebesgue measure in $\bbT$ (see \cite{Jitomirskaya1999Annals}). Note that $\mathcal{E},\mathcal{E}^{i}$ and $\tilde{\mathcal{E}}$ depend on $\theta$ for fixed $\Phi$. We will use $X(\theta),X\in\{\mathcal{E},\mathcal{E}^{i},\tilde{\mathcal{E}}\}$ to make such dependence explicit, yet, we sometimes suppress them from the notations to make things look concise.
	The main purpose of this subsection is to prove the following theorem:
	\begin{theorem}\label{t.mainThm2}
		Let $\tilde{\mathcal{E}}(\theta)$ be the associated extended CMV matrix of \eqref{eq.blockDef3}, and assume that $\Phi\in \DC$ is fixed and $L(z)>0$. If $\theta$ is \emph{non-resonant} with respect to $\Phi$,  then $\tilde{\mathcal{E}}(\theta)$ displays Anderson localization.
	\end{theorem}
	
	Once we have this,  Theorem \ref{thm:mobs} (b) follows as a consequence of Proposition \ref{prop:basicRelations} (4). To prove Theorem \ref{t.mainThm2}, it suffices to show that every generalized eigenfunction of $\tilde{\mathcal{E}}$ decays exponentially. 
	\begin{definition}
		We say that a nonzero sequence $\Psi:\bbZ\to\bbC$ is a \emph{generalized eigenfunction} of the  extended CMV matrix $\tilde{\mathcal{E}}$ with corresponding generalized eigenvalue $z\in\bbC$ if 
		\begin{equation*}
			\tilde{\mathcal{E}}\Psi=z\Psi
		\end{equation*}
		and there exist constants $M,N$ such that $|\Psi_{n}|\leq M(1+|n|)^{N}$, i.e., $\Psi$ is polynomially bounded.
	\end{definition}
	
	Schnol's theorem \cite{shnol1957behavior} asserts that the generalized eigenvalues sit in the spectrum and that they comprise \emph{spectrally} almost every $z$ in the spectrum. To formulate this precisely, one needs the following notion. It is well-known and not hard to check that for any $k$, $\{\delta_{2k},\delta_{2k+1}\}$ is a cyclic set for any CMV matrix with nonvanishing $\rho$'s. The reader may find a detailed proof for the CMV case in \cite[Lemma 3]{MungerOng2014JMP}, or in the more general matrix-valued version in \cite[Proposition VI.3.]{WeAreSchur}. Thus, the spectral measure $\mu_{\tilde{\CE}}^\univ$ given by
	\begin{equation}
		\int f(z) \, d\mu_{\tilde{\CE}}^\univ(z) =
		\langle \delta_0, f(\tilde{\CE}) \delta_0 \rangle
		+ \langle \delta_1, f(\tilde{\CE}) \delta_1 \rangle
	\end{equation}
	serves as a universal spectral measure of $\mathcal{E}$ in the sense that every other spectral measure of $\tilde{\CE}$ is absolutely continuous with respect to $\mu_{\tilde{\CE}}^\univ$.
	Then one has that \cite[Theorem 3.4]{DFLY2016DCDS}
	\begin{theorem}[Schnol's Theorem]\label{schnol}
		Let $\tilde{\CE}$ be an extended CMV matrix, $\mathcal{G}$  the set of its generalized eigenvalues, and $\sigma(\tilde{\mathcal{E}})$ its spectrum. Then we have the following:
		\begin{itemize}
			\item $\mathcal{G}\subset\sigma(\tilde{\mathcal{E}})$,
			\item $\mu_{\tilde{\CE}}^\univ(\sigma(\tilde{\mathcal{E}})\setminus\mathcal{G})=0$,
			\item $\overline{\mathcal{G}}=\sigma(\tilde{\mathcal{E}})$.
		\end{itemize}
	\end{theorem}
	
	To deduce the desired localization statements, the key is to prove the exponential decay of the Green's functions. Let us first introduce some necessary notations. Let $\Lambda=[a,b]\subset\bbZ$ be a finite interval and $\beta_{1},\beta_{2}\in\partial\bbD\cup\{\bullet\}$. Given $\{\alpha_{n}\}\subset\bbD$, define $\{\tilde{\alpha}_{n}\}\subset\overline{\bbD}$ as follows:
	\begin{equation}\label{eq.newVerblunsky}
		\tilde{\alpha}_{j}=\begin{cases}
			\beta_{1}& j=a-1,\\
			\alpha_{j}& j\neq a-1,b,\\
			\beta_{2}& j=b.
		\end{cases}
	\end{equation}
	
	Let $\mathcal{E}^{\beta_{1},\beta_{2}}$ be the GECMV matrix with Verblunsky coefficients $\{\tilde{\alpha}_{n},\rho_{n}\}$ and let $\chi_{\Lambda}$ be the projection onto $\Lambda$. 
	Define $\mathcal{E}_{\Lambda}^{\beta_{1},\beta_{2}}=\chi_{\Lambda}\mathcal{E}^{\beta_{1},\beta_{2}}\chi^{*}_{\Lambda}$. One can verify that $\mathcal{E}^{\beta_{1},\beta_{2}}_{\Lambda}$ is unitary when $\beta_{1},\beta_{2}\in\partial\bbD$. We will also use $\mathcal{E}_{\Lambda}^{\bullet,\beta_{2}}=\mathcal{E}_{\Lambda}^{\alpha_{a-1},\beta_{2}}$ and $\mathcal{E}^{\beta_{1},\bullet}_{\Lambda}=\mathcal{E}_{\Lambda}^{\beta_{1},\alpha_{b}}$ to denote the finite restrictions where the boundary conditions are chosen in $\partial\bbD$ on one side of $\Lambda$ and open on the other. The finite unitary restrictions $\mathcal{L}^{\beta_{1},\beta_{2}}_{\Lambda}$ and $\mathcal{M}^{\beta_{1},\beta_{2}}_{\Lambda}$ are defined in the same way.
	
	Let
	\begin{equation*}
		\rho_\Lambda=\prod_{j\in\Lambda}\rho_{j},
	\end{equation*}
	and define
	\begin{equation}\label{eq.Polynomial}
		P^{\beta_{1},\beta_{2}}_{z,\Lambda}=|\rho_\Lambda|^{-1}\det(z\idty-(\mathcal{E}^i)^{\beta_{1},\beta_{2}}_{\Lambda})
	\end{equation}
	for the GECMV matrix $\mathcal E^i$.
	For the case $a>b$, we just take $P^{\beta_{1},\beta_{2}}_{z,[a,b]}=1$. 
	Analogously, define
	\begin{equation}\label{eq.Polynomial1}
		\tilde{P}^{\beta_{1},\beta_{2}}_{z,\Lambda}=|\rho_\Lambda|^{-1}\det(z\idty-\tilde{\mathcal{E}}^{\beta_{1},\beta_{2}}_{\Lambda})
	\end{equation}
	for the CMV matrix $\mathcal E$. Then we have the following invariance property:
	
	\begin{lemma}\label{lem:sameDeterminant}
		Let $\mathcal{E}^{i},\tilde{\mathcal{E}},P^{\beta_{1},\beta_{2}}_{z,\Lambda}, \tilde{P}^{\beta_{1},\beta_{2}}_{z,\Lambda}$ be as above. Then 
		\begin{equation*}
			P^{\beta_{1},\beta_{2}}_{z,\Lambda}=\tilde{P}^{\beta_{1},\beta_{2}}_{z,\Lambda}.
		\end{equation*}
	\end{lemma}
	\begin{proof}
		This follows from two observations: (1) By Theorem~\ref{thm:rho_phases}, there exists a unitary diagonal gauge transformation $D$ such that $(\mathcal{E}^{i})^{\beta_{1},\beta_{2}}=D\tilde{\mathcal{E}}^{\beta_{1},\beta_{2}}D^{*}$. (2) $\chi_\Lambda$ and $D$ are both diagonal matrices, such that $\chi_\Lambda D=D\chi_\Lambda=(\chi_\Lambda D\chi_\Lambda)\chi_\Lambda=:D_\Lambda \chi_\Lambda$.
		Therefore,
		\begin{equation*}    
			\begin{aligned}
				\det(z\idty-(\mathcal{E}^{i})^{\beta_{1},\beta_{2}}_{\Lambda})
				=\det \chi_{\Lambda}D(z\idty-\tilde{\mathcal{E}}^{\beta_{1},\beta_{2}})D^{*}\chi^{*}_{\Lambda}
				=\det D_{\Lambda}(z\idty-\tilde{\mathcal{E}}^{\beta_{1},\beta_{2}}_{\Lambda})D^{*}_{\Lambda}.
			\end{aligned}
		\end{equation*}
		Multiplying with $|\rho_\Lambda|^{-1}$ on both sides concludes the proof .
	\end{proof}
	
	We emphasize that this equivalence is crucial since it connects the current setting of GECMV matrices to the existing theory for extended CMV matrices in that it allows us to transfer statements about $\tilde P^{\beta_{1},\beta_{2}}_{z,\Lambda}$ to $P^{\beta_{1},\beta_{2}}_{z,\Lambda}$. 
	More concretely, our main application of Lemma \ref{lem:sameDeterminant} will be to write down the exact relations \eqref{eq.GreenToTransfer} and \eqref{eq.GreentoTransfer1} below for GECMV matrices. 
	This is important, since $\mathcal E^i$ possesses a reflection symmetry which will allow us to conclude the evenness of $\det(z\idty-(\mathcal{E}^{i})^{\beta_{1},\beta_{2}}_{[1,4k-2]})(\theta)$. We henceforth do not distinguish explicitly \eqref{eq.Polynomial} from \eqref{eq.Polynomial1}, and in slight abuse of notation just write  $P^{\beta_{1},\beta_{2}}_{z,\Lambda}$.
	
	Recalling $\tilde{\gamma}$ given in \eqref{eq.scalarFactorCont}, we need the following estimates:
	\begin{lemma}\label{lem:upperBoundRho}
		For any $\eta>0$, there exists $N>0$ such that for any $n>N$
		\begin{equation*}
			e^{n(\tilde{\gamma}/4-\eta)}\leq\prod_{j=0}^{n-1}|\rho_{j}|\leq e^{n(\tilde{\gamma}/4+\eta)}.
		\end{equation*}
	\end{lemma}
	\begin{proof}
		Since $\rho_{4j-1}(\theta)=\lambda_{2}\cos2\pi(\theta+2j\Phi)-i\lambda_{2}'$ and $\rho_{2j}=\lambda_{1},\rho_{4j+1}=-i$ by \eqref{eq:mosaicVerblunskies}, and $\theta\to\theta+2\Phi$ is ergodic in $\bbT$, by the Ergodic Theorem and \eqref{eq:lyap_scalar_contribution},
		\begin{equation*}
			\lim\limits_{n\to\infty}\frac{1}{n}\log\left[\prod_{j=1}^{n}|\rho_{4j-1}(\theta)\rho_{4j-2}\rho_{4j-3}\rho_{4j-4}|\right]=\int_{\bbT}\log|\rho_{3}(\theta)\rho_{2}\rho_{1}\rho_{0}|d\theta=\tilde{\gamma}.
		\end{equation*}
		For any $n>0$, let $l$ be the integer such that $n=4l+r$ with $0\leq r\leq 3.$ Then 
		\begin{equation*}
			\prod_{j=0}^{n-1}|\rho_{j}|=\begin{dcases}
				\prod_{k=1}^{l}|\rho_{4k-1}(\theta)\rho_{4k-2}\rho_{4k-3}\rho_{4k-4}| & r=0,\\
				\prod_{i=0}^{r-1}|\rho_{4l+i}|\prod_{k=1}^{l}|\rho_{4k-1}(\theta)\rho_{4k-2}\rho_{4k-3}\rho_{4k-4}| & r=1,2,3.
			\end{dcases}  
		\end{equation*}
		Therefore, we have 
		\begin{equation*}
			\frac{1}{n}\left|\log\prod_{j=0}^{n-1}|\rho_{j}|-\log\prod_{k=1}^{l}|\rho_{4k-1}(\theta)\rho_{4k-2}\rho_{4k-3}\rho_{4k-4}|\right|\leq \frac{\log|\rho_{2}\rho_{1}\rho_{0}|}{n}=\frac{|\log\lambda_{1}^{2}|}{n}
		\end{equation*}
		Thus for any $\eta>0$, there exists $N>0$ such that for any $n>N$
		\begin{equation*}
			\left|\frac{1}{n}\log\prod_{j=0}^{n-1}|\rho_{j}|-\frac{\tilde{\gamma}}{4}\right|\leq\eta.
		\end{equation*}
	\end{proof}
	
	Consider the equation $(z\idty-\tilde{\mathcal{E}})\Psi^{z}=0$. Since $\tilde{\mathcal{E}}=\tilde{\mathcal{L}}\tilde{\mathcal{M}}$, it follows equivalently that $(z\tilde{\mathcal{L}}^{*}-\tilde{\mathcal{M}})\Psi^{z}=0$. 
	Define the finite-volume \emph{Green's function} as
	\begin{equation}\label{eq.GreenFunction}
		\tilde{G}_{z,\Lambda}^{\beta_{1},\beta_{2}}:=\left(z(\tilde{\mathcal{L}}^{\beta_{1},\beta_{2}}_{\Lambda})^{*}-\tilde{\mathcal{M}}^{\beta_{1},\beta_{2}}_{\Lambda}\right)^{-1},
	\end{equation}
	and denote its matrix elements for $x,y\in\Lambda$ by $\tilde{G}_{z,\Lambda}^{\beta_{1},\beta_{2}}(x,y):=\langle \delta_{x},\tilde{G}_{z,\Lambda}^{\beta_{1},\beta_{2}}\delta_{y}\rangle$. Then, by \cite[Lemma 3.9]{krugerOrthogonalPolynomialsUnit2013} we have for $a<y<b$,
	\begin{equation}\label{eq.solGreen}
		\Psi^{z}(y)=\tilde{G}^{\beta_{1},\beta_{2}}_{z,\Lambda}(y,a)\tilde{\Psi}^{z}(a)+\tilde{G}^{\beta_{1},\beta_{2}}_{z,\Lambda}(y,b)\tilde{\Psi}^{z}(b),
	\end{equation}
	where the values at the endpoints of $\Lambda=[a,b]$ are given by
	\begin{equation*}
		\tilde{\Psi}^{z}(a)=\begin{cases}
			(z\overline{\beta_{1}}-\alpha_{a})\Psi^{z}(a)-|\rho_{a}|\Psi^{z}(a+1),& a\text{ is even,}\\
			(z\alpha_{a}-\beta_{1})\Psi^{z}(a)+z|\rho_{a}|\Psi^{z}(a+1), &a\text{ is odd,}
		\end{cases}
	\end{equation*}
	and
	\begin{equation*}
		\tilde{\Psi}^{z}(b)=
		\begin{cases}
			(z\overline{\beta_{2}}-\alpha_{b})\Psi^{z}(b)-|\rho_{b}|\Psi^{z}(b-1),& b\text{ is even,}\\
			(z\alpha_{b}-\beta_{2})\Psi^{z}(b)+z|\rho_{b-1}|\Psi^{z}(b-1),&b\text{ is odd.}
		\end{cases}
	\end{equation*}
	By \cite[Proposition 3.8]{krugerOrthogonalPolynomialsUnit2013} and the correction note in \cite[Appendix B.1.]{ZhuLocalizationRandomCMV}, 
	\begin{equation}\label{eq.GreenToPolynomial}
		\left|\tilde{G}^{\beta_{1},\beta_{2}}_{z,\Lambda}(x,y)\right|=\frac1{|\rho_{y}|}\left|\frac{P^{\beta_{1},\bullet}_{z,[a,x-1]}P^{\bullet,\beta_{2}}_{z,[y+1,b]}}{P^{\beta_{1},\beta_{2}}_{z,\Lambda}}\right|,\quad x,y\in\Lambda.
	\end{equation}
	The next step is to connect the Green's function to the Szeg\H{o} transfer matrix.
	Let \begin{equation}\label{eq.Zego}\tilde S_{n,z}=\frac{1}{|\rho_{n}|}\begin{bmatrix}z&-\overline{\alpha_{n}}\\-\alpha_{n}z&1\end{bmatrix}\end{equation} be the Szeg\H{o} cocycle map of the extended CMV matrix $\tilde{\mathcal{E}}$ which has all $\rho$'s real.
	By \cite[Corollary 3.11]{krugerOrthogonalPolynomialsUnit2013} and Lemma \ref{lem:sameDeterminant} we have that
	\begin{equation}\label{eq.GreenToTransfer}
		\begin{bmatrix}P^{\beta_{1},\beta_{2}}_{z,[a,b]}&P^{-\beta_{1},\beta_{2}}_{z,[a,b]}\\P^{\beta_{1},-\beta_{2}}_{z,[a,b]}&P^{-\beta_{1},-\beta_{2}}_{z,[a,b]}\end{bmatrix}=
		\begin{bmatrix}z&-\overline{\beta_{2}}\\z&\overline{\beta_{2}}\end{bmatrix}\left(\frac{1}{|\rho_{b}|}\prod_{j=a}^{b-1}\tilde{S}_{j,z}\right)\begin{bmatrix}1&1\\\beta_{1}&-\beta_{1}\end{bmatrix}
	\end{equation}
	and
	\begin{equation}\label{eq.GreentoTransfer1}
		\begin{bmatrix}P^{\beta,\bullet}_{z,\Lambda}\\P^{\bullet,\beta}_{z,\Lambda}\end{bmatrix}=\prod_{j=a}^{b}\tilde{S}_{j,z}\begin{bmatrix}1\\\overline{\beta}\end{bmatrix}.
	\end{equation}
	It follows that 
	\begin{equation}\label{eq.PolynomialUpperBound}
		|P^{\beta_{1},\bullet}_{z,[a,x-1]}|\leq \sqrt{2}\:\Bigg\Vert \prod_{j=a}^{x-1}\tilde{S}_{j,z}\Bigg\Vert,
		\qquad
		|P^{\bullet,\beta_{2}}_{z,[y+1,b]}|\leq\sqrt{2}\:\Bigg\Vert\prod_{j=y+1}^{b}\tilde{S}_{j,z}\Bigg\Vert.
	\end{equation}

	\begin{definition}\label{def.regular}
		Fix $z=e^{it}\in\partial\bbD$, $\gamma\in\bbR$ and $\bar k\in\bbZ$.
		We say that $y\in\bbZ$ is \emph{$(\gamma,\bar k)$-regular} if 
		\begin{itemize}
			\item there exists $[n_{1},n_{2}]$ containing $y$ such that $n_{2}=n_{1}+\bar k-1$, that is, there is an interval of size $\bar k$ that contains $y$,
			\item $|y-n_{i}|\geq \frac{\bar k}{7}$, $i=1,2$, that is, the distance of $y$ to the boundary of $[n_1,n_2]$ is at least $\bar k/7$,
			\item $|\tilde{G}_{z,[n_{1},n_{2}]}^{\beta_{1},\beta_{2}}(y,n_{i})|<e^{-\gamma |y-n_{i}|}$, $i=1,2$, that is, the Green's function decays exponentially with a rate at least $\gamma$.
		\end{itemize}
		Otherwise, we call $y\in\bbZ$ \emph{$(\gamma,\bar k)$-singular}.
	\end{definition}
	It is well known that if $\Psi^{z}\neq0$ is a non-zero generalized eigenfunction, then any $y$  with $\Psi^{z}(y)\neq 0$ is $(\gamma,\bar k)$-singular for sufficiently large $\bar k$. Thus, we usually assume $\Psi^{z}(0)\neq 0$, and replace $\Psi^{z}(0)$ with $\Psi^{z}(1)$ otherwise.
	
	To prove the exponential decay of the generalized eigenfunction corresponding to the generalized eigenvalue $z$, we need the following lemma which guarantees that there exists a $\bar k$ for which $y$ is close to being $(\gamma,\bar k)$-regular:
	\begin{lemma}\label{regularPoint}
		Suppose that $\Phi\in \DC$, $\theta$ is non-resonant w.r.t. $\Phi$ and $L(z)>0$. Then for any $\epsilon>0$ and $|y|>k_0(\theta,\Phi,z,\epsilon)$ large enough, there exists $\bar k>\frac{5|y|}{16}$, such that $y$ is $(L(z)/2-\epsilon,\bar k)$-regular.
	\end{lemma}
	We prove this lemma at the end of this section. In the following, we will take $\gamma\equiv L(z)/2$ for simplicity, and we will fix $z$ and $\theta$ and suppress them from the notation. Assuming that Lemma \ref{regularPoint} holds, we can prove Theorem \ref{t.mainThm2}.
	
	\begin{proof}[Proof of Theorem \ref{t.mainThm2}]\label{p.localizationProof}
		By Schnol's theorem (Theorem \ref{schnol}), it is enough to prove that any generalized eigenfunction decays exponentially.
		For $|y|>k_{0}$, since $y$ is $(\gamma-\epsilon,\bar k)-$regular by Lemma \ref{regularPoint}, we have 
		\begin{equation*}
			|\tilde{G}^{\beta_{1},\beta_{2}}_{z,[n_{1},n_{2}]}(y,n_{i})|<e^{-(\gamma-\epsilon)|y-n_{i}|}\leq e^{-\frac{\gamma-\epsilon}{7}\frac{5|y|}{16}}.
		\end{equation*}
		Since $|\Psi^{z}(y)|\leq M(1+|y|)^{N}$ for any $y$, we obtain the following estimate from \eqref{eq.solGreen}:
		\begin{align*}
			|\Psi^{z}(y)|&\leq 2e^{-(\gamma-\epsilon)(y-n_{1})}\max\{|\Psi^{z}(n_{1})|,|\Psi^{z}(n_{1}+1)|\} \\
			&\qquad + 2e^{-(\gamma-\epsilon)(n_{2}-y)}\max\{|\Psi^{z}(n_{2})|,|\Psi^{z}(n_{2}-1)|\}\\
			&\leq 2\left(e^{-(\gamma-\epsilon)(y-n_{1})}M(1+|n_{1}|)^{N}+e^{-(\gamma-\epsilon)(n_{1}-y)}M(1+|n_{2}|^{N})\right).
		\end{align*}
		Since $|n_{i}|\leq |n_{i}-y|+|y|$ for $i=1,2$, we have 
		\begin{equation*}
			(1+|n_{i}|^{N})\leq 2^{N}\max\{|y|^{N},|n_{i}-y|^{N}\},
		\end{equation*}
		such that in both cases we have the exponential decay estimate
		\begin{equation*}
			|\Psi^{z}(y)|\leq e^{-(\gamma-\epsilon)\frac{5|y|}{14\times16}}.
		\end{equation*}
		From this the result follows.
	\end{proof}
	
	It remains to to prove Lemma \ref{regularPoint}, which we do in a sequence of steps. 
	More concretely, we need to establish the exponential upper bound on the absolute value of the Green's function $\tilde{G}^{\beta_{1},\beta_{2}}_{z,[a,b]}$ in the definition of regularity (Definition \ref{def.regular}). To this end, we consult \eqref{eq.GreenToPolynomial} and bound the numerator from above and the denominator from below.
	
	Define $\tilde{S}^{++}_{n,z}$ as the four-step quasi-periodic cocycle of $\tilde{S}_{n,z}$ in the same fashion as in \eqref{eq.combinedSzego}. It is readily verified that for $z\in\partial\bbD$,
	we have $L(2\Phi,\tilde{S}^{++}_{z})=L(2\Phi,S^{++}_{z})=2L(z)$, where $L(z)$ is the Lyapunov exponent given in Theorem \ref{Lyapunov}.
	\begin{lemma}\label{lem:PolynomialUpperBound}
		For any $\epsilon>0$, $z\in\partial\bbD$, there exists $k_{1}=k_{1}(\epsilon,z)$ such that
		\begin{equation*}
			\left|P^{\beta_{1},\bullet}_{z,[a,b]}\right|,\:\left|P^{\bullet,\beta_{2}}_{z,[a,b]}\right|< e^{(\gamma+\epsilon)(b-a+1)}
		\end{equation*}
		if $b-a+1>k_{1}$.
	\end{lemma}
	
	\begin{proof}
		By \eqref{eq:twoblockAzdef}, \eqref{eq.ConjugatedTransferMatrix} and \eqref{eq.combinedSzego}, we have 
		\begin{equation*}
			L(z)=\lim_{k\to\infty}\frac{1}{2k}\int_{\bbT}\log\Bigg\Vert\prod_{j=k-1}^{0}A^{+}_{z}(\theta+2j\Phi)\Bigg\Vert d\theta=\lim_{k\to\infty}\frac{1}{2k}\int_{\bbT}\log\Bigg\Vert\prod_{j=k-1}^{0}S_{z}^{++}(\theta+2j\Phi)\Bigg\Vert d\theta.
		\end{equation*}
		Therefore 
		\begin{equation*}
			L(z)=\lim\limits_{k\to\infty}\frac{1}{2k}\int_{\bbT}\log\Bigg\Vert\prod_{j=k-1}^{0}\tilde{S}^{++}_{z}(\theta+2j\Phi)\Bigg\Vert d\theta.
		\end{equation*} 
		Since $\tilde{S}^{++}_{z}$ are four combined steps of $\tilde{S}_{j,z}$ this implies that $\lim_{k\to\infty}k^{-1}\int_{\bbT}\log\Vert \prod_{j=k-1}^{0}\tilde{S}_{j,z}(\theta)\Vert d\theta=L(z)/2=\gamma.$
		Then the statement follows from Furman's well-known result \cite{Furman1997} and \eqref{eq.PolynomialUpperBound}.
	\end{proof}

	Recall that $P^{\beta_{1},\beta_{2}}_{z,\Lambda}$ in \eqref{eq.Polynomial} and \eqref{eq.Polynomial1} depends on the phase parameter $\theta$. In the following, we will write $P^{\beta_{1},\beta_{2}}_{z,\Lambda}(\theta)$ to make this dependence explicit.
	Now let us give the lower bound on the denominator in \eqref{eq.GreenToPolynomial} by following the idea of \cite{JitomirskayaYangETDS}. To facilitate the proof, we restrict to the concrete interval $[1,4k-2]$ which is the only interval we shall require:
	\begin{lemma}\label{lem.Herman}
		For any $\epsilon>0$, $z\in\partial\bbD$, there exists $k_{2}=k_{2}(\epsilon,z)> 0$ such that 
		\begin{equation*}
			\frac{1}{4k-2}\int_{\bbT}\log|P^{\beta_{1},\beta_{2}}_{z,[1,4k-2]}(\theta)|d\theta\geq \gamma-\epsilon
		\end{equation*}
		for any $4k-2>k_{2}$.
	\end{lemma}
	\begin{proof}
		Let $D(\theta)=|\rho_{3}\rho_{2}\rho_{1}\rho_{0}|\tilde{S}_{3,z}\tilde{S}_{2,z}\tilde{S}_{1,z}\tilde{S}_{0,z}$ with $\tilde{S}_{n,z}$ given by \eqref{eq.Zego}. Direct computation gives 
		\begin{equation*}
			D(\theta)=z^{2}\begin{bmatrix}
				z^{2}+\lambda_{1}'^{2}+\lambda_{1}'\lambda_{2}s(\theta')(z+z^{-1})&-\lambda_{1}'(z+z^{-1})-\lambda_{2}s(\theta')(\lambda_{1}'^{2}+z^{-2})\\
				-\lambda_{1}'(z+z^{-1})-\lambda_{2}s(\theta')(\lambda_{1}'^{2}+z^{2})&\lambda_{1}'^{2}+z^{-2}+\lambda_{1}'\lambda_{2}s(\theta)(z+z^{-1})
			\end{bmatrix},
		\end{equation*}
		where we recall from \eqref{eq:cstwopi} the notation $s(\theta)=\sin2\pi(\theta)$ and we set $\theta'=\theta+2\Phi$. Writing $z=e^{it}$ and $s(\theta')=(e^{2\pi i\theta'}-e^{-2\pi i\theta'})/(2i)=(w-w^{-1})/(2i)$, where $w=e^{2\pi i\theta'}$,  then $D(\theta)$ can be written as
		\begin{equation*}
			\begin{aligned}
				D(\theta)&=w^{-1}e^{2it}\begin{bmatrix}
					(\lambda_{1}'^{2}+z^{2})w+2\lambda_{1}'\lambda_{2}\cos(t) \frac{w^{2}-1}{2i}&-2\lambda_{1}'w\cos(t)-\lambda_{2}(\lambda_{1}'^{2}+z^{-2}) \frac{w^{2}-1}{2i})\\
					-2\lambda_{1}'w\cos(t)-\lambda_{2}(\lambda_{1}'^{2}+z^{2})\frac{w^{2}-1}{2i}&(\lambda_{1}'^{2}+z^{-2})w+2\lambda_{1}'\lambda_{2}\cos(t)\frac{w^{2}-1}{2i}
				\end{bmatrix}\\
				&=:w^{-1}e^{2it}\hat{D}(w).
			\end{aligned}
		\end{equation*}
		Let $\tilde{\gamma}$ be given by \eqref{eq.scalarFactorCont} and
		let 
		\begin{equation*}
			V=\begin{bmatrix}1&0\\0&0\end{bmatrix},\quad B_{1}=\begin{bmatrix}1&1\\\beta_{1}&-\beta_{1}\end{bmatrix},\quad B_{2}=\begin{bmatrix}z&-\overline{\beta_{2}}\\z&\overline{\beta_{2}}\end{bmatrix}.
		\end{equation*}
		Then, by \eqref{eq.GreenToTransfer},
		\begin{equation}\label{eq.midCal}
			\begin{aligned}
				\int_{\bbT}\log|P^{\beta_{1},\beta_{2}}_{z,[1,4k-2]}(\theta)|d\theta+(k-1)\tilde{\gamma}
				&=\int_{\bbT}\log\Bigg\Vert VB_{2}\tilde{S}_{4k-3,z}\prod_{j=0}^{k-2}D(\theta+j2\Phi)\tilde{S}_{0,z}^{-1}B_{1}V\Bigg\Vert d\theta \\
				&=\int_{\partial\bbD}\log\Bigg\Vert V B_{2}\tilde{S}_{4k-3,z} \prod_{j=0}^{k-2}w^{-1}e^{2it}\hat{D}(we^{2ij2\Phi})\tilde{S}_{0,z}^{-1}B_{1}V\Bigg\Vert dw\\
				&\geq\log\Bigg\Vert VB_{2}\tilde{S}_{4k-3,z}\prod_{j=0}^{k-2}\hat{D}(0)\tilde{S}^{-1}_{0,z}B_{1}V\Bigg\Vert.
			\end{aligned}
		\end{equation}
		The last inequality is due to subharmonicity.
		By \eqref{eq.blockDef3}, 
		\begin{equation*}
			\tilde{S}_{4k-3,z}=i\begin{bmatrix}z&0\\0&1\end{bmatrix},\qquad\tilde{S}_{0,z}=\frac{1}{i\lambda_{1}}\begin{bmatrix}z&-\lambda_{1}'\\-\lambda_{1}'z&1\end{bmatrix},
		\end{equation*}
		and 
		\begin{equation*}
			\begin{aligned}\hat{D}(0)&=\frac{\lambda_{2}}{2i}\begin{bmatrix}
					-2\lambda_{1}'\cos(t)&\lambda_{1}'^{2}+z^{-2}\\
					\lambda_{1}'^{2}+z^{2}&-2\lambda_{1}'\cos(t)
				\end{bmatrix}
				=\frac{\lambda_{2}}{2i}Q^{-1}\begin{bmatrix}\lambda_+&0\\0&\lambda_-\end{bmatrix}Q
			\end{aligned}
		\end{equation*}
		where $Q$ is the normalized diagonalization matrix ($\det Q=1$), and
		\begin{equation*}
			\lambda_{\pm}=2\lambda_{1}'\cos t\pm\sqrt{\lambda_{1}^{2}+4\lambda_{1}'^{2}\cos^{2}t}.
		\end{equation*}
		Direct computations give the right side of \eqref{eq.midCal} as a linear combination of the log of $(\frac{\lambda_{2}}{2}\lambda_{\pm})^{k-1}$
		with non-zero constant coefficients (independent of $k$). Thus, for any $\epsilon>0$ and $4k-2>k_{2}(\epsilon,z)$ large enough, we have 
		\begin{equation*}
			\begin{aligned}
				\frac{1}{4k-2}\int_{\bbT}\log|P_{z,[1,4k-2]}(\theta)|d\theta&\geq\frac{1}{4}\left(\log\frac{\lambda_{2}(2\lambda_{1}'|\cos t|+\sqrt{\lambda_{1}^{2}+4\lambda_{1}'^{2}\cos^{2}t})}{2}-\tilde{\gamma}\right)-\epsilon\\
				&=\gamma-\epsilon.
			\end{aligned}
		\end{equation*}
	\end{proof}
	Let us denote
	\begin{equation*}
		\Gamma^{\beta_{1},\beta_{2}}_{z,\Lambda}:=\det(z\idty-(\mathcal{E}^{i})^{\beta_{1},\beta_{2}}_{\Lambda})\equiv|\rho_\Lambda|P^{\beta_{1},\beta_{2}}_{z,\Lambda}
	\end{equation*}
	to facilitate the statement of the results. In deriving the following statement, we want to use the reflection symmetry. Since we want to use the results of Section \ref{sec:reflection}, we fix the boundary conditions $\beta_2=\overline{\beta_1}$ for \eqref{eq.newReflectionProp} to hold. 
	The key ingredient of proving localization is the following lemma: 
	\begin{lemma}\label{lem:evenness}
		$\Gamma^{\beta_{1},\beta_{2}}_{z,[1,4k-2]}(\theta+\frac{1}{4})$ is a polynomial of $\cos2\pi(\theta+ (k-1)\Phi)$ of degree at most $k$.
	\end{lemma}
	
	\begin{proof}
		First note that $\Gamma^{\beta_{1},\beta_{2}}_{z,[1,4k-2]}(\theta)$ is a polynomial in $\sin2\pi\theta$ and $\cos2\pi\theta$ of degree at most $k$.
		This is a direct consequence of \eqref{eq.Polynomial1}, Lemma \ref{lem:sameDeterminant}, \eqref{eq.GreenToTransfer} and our specific model \eqref{eq.blockDef3}.
		
		Next, we show that $\Gamma_{z,[1,4k-2]}^{\beta_{1},\beta_{2}}(\theta+\frac{1}{4}-(k-1)\Phi)$ is an even function of $\theta$: Recall the reflection symmetry of the $\alpha_j$ and $\rho_j$ established in \eqref{eq:alpha_i_sym} and \eqref{eq:rho_i_sym}, respectively. 
		Then, evenness in $\theta$ follows from applying Proposition \ref{prop:deterEveness1} to $\Gamma_{z,[1,4k-2]}^{\beta_{1},\beta_{2}}(\theta+\frac{1}{4}-(k-1)\Phi)$. The change of variable $\theta\mapsto\theta+(k-1)\Phi$ concludes the proof.
	\end{proof}

	By Lemma \ref{lem:evenness}, there exists a polynomial $Q_{k}$ of degree $k$ such that $\Gamma^{\beta_{1},\beta_{2}}_{z,[1,4k-2]}(\theta+\frac{1}{4})=Q_{k}(\cos2\pi(\theta+(k-1)\Phi))$.
	For any positive integer $k$ and $r>0$, define 
	\begin{equation*}
		A_{k}^{r}=\{\theta\in\bbT:|Q_{k}(\cos2\pi\theta)|\leq e^{kr}\}.
	\end{equation*}
	Define 
	\begin{equation}\label{eq.newExponent}
		\gamma'=\gamma+\tilde{\gamma}/4.
	\end{equation} 
	Then, similar to \cite[Lemma 6]{Jitomirskaya1999Annals}, the following holds:
	\begin{lemma}\label{singularCluster}
		Assume that $y$ is $(\gamma-\epsilon,4k-2)$-singular for some $k\in \bbZ$ and $\epsilon>0$. Then for each $j\in\bbZ$ with 
		\begin{equation*}
			y-\left\lfloor\frac{2}{4}(4k-2)\right\rfloor+(k-1)\leq 2j\leq y+\left\lfloor\frac{1}{4}(4k-2)\right\rfloor+(k-1) ,
		\end{equation*} 
		we have $\theta+ 2j\Phi\in A^{4\gamma'-\frac{\epsilon}{8}}_{k}$, where $\gamma'$ is given by \eqref{eq.newExponent}, provided $4k-2>k_{3}(\gamma,\frac{\epsilon}{48})$ is sufficiently large.
	\end{lemma}
	\begin{proof}
		Let us take $k_{3}=\max\{k_{0},k_{1},k_{2},N\}$ with $ N, k_0$, $k_1$, $k_2$ from Lemma \ref{lem:upperBoundRho}, Lemma \ref{regularPoint}, Lemma \ref{lem:PolynomialUpperBound} and Lemma \ref{lem.Herman}, respectively. Then, by Lemma \ref{lem:PolynomialUpperBound}, for any $\epsilon'>0$, $b-a+1>k_{3}$
		\begin{equation}\label{eq.PupperB}
			\left|P^{\beta_{1},\bullet}_{z,[a,b]}\right|,\left|P^{\bullet,\beta_{2}}_{z,[a,b]}\right|< e^{|b-a+1|(\gamma+\epsilon')}.
		\end{equation}
		Since $y$ is $(\gamma-\epsilon,4k-2)$-singular, then without loss of generality, for any $n_{1}<n_{2}$ such that $y\in [n_{1},n_{2}], n_{2}-n_{1}+1=4k-2$ with $|y-n_{i}|\geq \frac{4k-2}{7}$, we  assume 
		\begin{equation*}
			\left|\tilde{G}^{\beta_{1},\beta_{2}}_{z,[n_{1},n_{2}]}(y,n_{1})\right|>e^{-|y-n_{1}|(\gamma-\epsilon)}.
		\end{equation*}
		Suppose that there exists some $j_0$ with 
		\begin{equation*}
			y-\left\lfloor\frac{2}{4}(4k-2)\right\rfloor+(k-1)\leq 2j_{0}\leq y+\left\lfloor\frac{1}{4}(4k-2)\right\rfloor+(k-1) 
		\end{equation*}
		such that $\theta+2j_{0}\Phi\notin A_{k}^{4\gamma'-\frac{\epsilon}{8}}$, that is, $|Q_{k}(\cos 2\pi(\theta+2j_{0}\Phi))|>e^{k(4\gamma'-\frac{\epsilon}{8})}$. Let $\tilde{\theta}=\theta+(j_{0}-(k-1))\Phi$. Then $|\Gamma^{\beta_{2},\beta_{2}}_{z,[1,4k-2]}(\tilde{\theta})|>e^{k(4\gamma'-\frac{\epsilon}{8})}$ by Lemma \ref{lem:evenness}. By \eqref{eq.GreenToPolynomial} and \eqref{eq.PupperB}, we conclude that 
		\begin{equation*}
			|\tilde{G}^{\beta_{1},\beta_{2}}_{z,[n_{1},n_{2}]}(y,n_{1})|<|\rho_{y}|^{-1}\left(\prod_{j=n_{1}}^{n_{2}}|\rho_{j}|\right)e^{(\gamma+\epsilon')(n_{2}-y)}e^{-k(4\gamma'-\frac{\epsilon}{8})}.
		\end{equation*}
		By Lemma \ref{lem:upperBoundRho}, if $4k-2>N$, we have 
		\begin{equation*}
			\prod_{j=n_{1},j\neq y}^{n_{2}}|\rho_{j}|\leq e^{(n_{2}-n_{1})(\tilde{\gamma}/4+\eta)}.
		\end{equation*}
		Putting the above inequalities together yields 
		\begin{equation*}
			|\tilde{G}^{\beta_{1},\beta_{2}}_{z,[n_{1},n_{2}]}(y,n_{1})|<e^{|y-n_{1}|(\gamma-\epsilon)}
		\end{equation*}
		whenever we take $28\eta+24\epsilon'<\frac{\epsilon}{9}$.  This contradicts the $(\gamma-\epsilon,4k-2)$-singularity of $y$. 
	\end{proof}
	
	We can write a  polynomial $Q_k(x)$ of degree $k$ in the following Lagrange interpolation form 
	\begin{equation}\label{langrange}
		Q_{k}(x)=\sum_{j=0}^{k} Q_{k}(\cos2\pi\theta_{j})\frac{\prod_{i\neq j} (x-\cos2\pi\theta_{i})}{\prod_{i\neq j}(\cos2\pi\theta_{j}-\cos2\pi\theta_{i})}.
	\end{equation}
	\begin{definition}[$\epsilon$-uniform]
		The set $\{\theta_{j}\}_{j=0}^{k}\subset\bbT$ is called \emph{$\epsilon$-uniform} if and only if 
		\begin{equation} \label{eq:epsunif}
			\max_{x\in[-1,1]}\max_{0\leq i\leq k}\prod_{j=0,j\neq i}^{k}\left|\frac{x-\cos2\pi\theta_{j}}{\cos2\pi\theta_{i}-\cos2\pi\theta_{j}}\right|<e^{k\epsilon}.\end{equation} 
	\end{definition}
	Then the following result holds:
	\begin{lemma}\label{notUniform}
		Let $0<\epsilon'<\epsilon$, $k\in\bbN_{+},$ and $\gamma>0$. If $\{\theta_{0},\cdots,\theta_{k}\}\subset A_{k}^{4\gamma'-\epsilon}$, then $\{\theta_{0},\cdots,\theta_{k}\}$ can not be $\epsilon'-$uniform for any sufficiently large $k$ such that $4k-2>k_{4}(\epsilon,\epsilon')$.
	\end{lemma}
	\begin{proof}
		
		If $\{\theta_{0},\cdots,\theta_{k}\}\subset A_{k}^{4\gamma'-\epsilon}$ is $\epsilon'$-uniform, then as a result of \eqref{langrange}, we have the following estimates: \begin{equation}\label{eq.Pestimate}|\Gamma^{\beta_{1},\beta_{2}}_{z,[1,4k-2]}(\theta+\frac{1}{4})|=|Q_{k}(\cos 2\pi(\theta+(k-1)\Phi)|\leq (k+1)e^{k(4\gamma'-(\epsilon-\epsilon'))}.\end{equation}
		Compare this with Lemma \ref{lem:upperBoundRho} and Lemma \ref{lem.Herman}:
		\begin{equation}\label{eq.PlowerBound}
			\frac{1}{4k-2}\int\log |P^{\beta_{1},\beta_{2}}_{z,[1,4k-2]}(\theta)|d\theta\geq \gamma-\epsilon''
		\end{equation}
		where $\epsilon''$ is arbitrarily small and $k$ sufficiently large. Then \eqref{eq.Pestimate} and \eqref{eq.PlowerBound} lead to a contradiction if we pick $\epsilon''<\frac{\epsilon-\epsilon'-4\eta}{5}$.
	\end{proof}
	Let $p_{n}/q_{n}$ be the sequence of continued fraction approximants of $2\Phi$, let $y$ be large enough, let $m$ be such that $q_{m}\leq \frac{y}{16}<q_{m+1}$ and let $s$ be the largest positive integer with $sq_{m}<\frac{y}{16}$. Define
	\begin{equation*}
		I_{1}=[0,sq_{m}]\cap\bbZ,\qquad I_{2}=\left[1+\left\lfloor\frac{y}{2}\right\rfloor-sq_{m},\left\lfloor\frac{y}{2}\right\rfloor+sq_{m}\right]\cap\bbZ.
	\end{equation*}
	Note that $I_{1}\cap I_{2}$ is empty, the elements of $\{\theta+2j\Phi\}_{j\in I_{1}\cup I_{2}}$ are all distinct, and the number of points in $I_{1}\cup I_{2}$ is $3sq_{m}+1$. Actually, the elements of $\{\cos2\pi (\theta+2j\Phi)\}_{j\in I_{1}\cup I_{2}}$ are also distinct.  Moreover the following  property holds: 
	\begin{lemma}\label{uniformlyDistri}
		For any $\epsilon>0$, the set
		$\{\theta+2j\Phi\}_{j\in I_{1}\cup I_{2}}$ is $\epsilon$-uniform for $y>y_{0}(\Phi,\theta,\epsilon)$ sufficiently large.
	\end{lemma}
	The proof is standard, and we thus leave it to Appendix \ref{uniform}. Once we have this, we are  ready to prove Lemma \ref{regularPoint}:
	\begin{proof}[Proof of Lemma \ref{regularPoint}]
		Let $K=\max\{k_{1},k_{2},k_{3},k_{4}\}$ and let $y$ be sufficiently large such that $12sq_{m}-2>K$. By Lemma~\ref{notUniform} and Lemma \ref{uniformlyDistri}, $\{\theta+2j\Phi\}_{j\in I_{1}\cup I_{2}}$ cannot be inside the set  $A_{3sq_{m}}^{4\gamma'-\frac{\epsilon}{8}}$. Since $0$ is $(\gamma-\epsilon,12sq_{m}-2)$-singular by the assumption $\Psi^{z}(0)\neq 0$, $y$ must be $(\gamma-\epsilon,12sq_{m}-2)$-regular, since if $y$ would be $(\gamma-\epsilon,12sq_{m}-2)$-2singular, the clusters of points given by Lemma \ref{singularCluster} with respect to $0$ and $y$ would cover $I_{1}$ and $I_{2}$. This is a contradiction. Notice also that $12sq_{m}-2>\frac{5y}{16}$, so the proof can be completed.
	\end{proof}

	\subsection{Decay Rate of Eigenfunctions}
	
	In this section we prove that the eigenfunctions in the supercritical case decay at the Lyapunov rate: 
	\begin{theorem} \label{t:efctDecay}
		For Diophantine $\Phi$, non-resonant $\theta$, and every eigenvalue $z\in I_\pp$ of 
		$\mathcal E(\theta)$ as in Theorem \ref{thm:mobs}, the corresponding  eigenfunction $\Psi^z=(\cdots,\Psi^{z}(2n),\Psi^{z}(2n+1),\cdots)$ satisfies
		\begin{equation}
			\lim_{n \to \infty} \frac{1}{n} \log(|\Psi^{z}(2n)|^2 +  |\Psi^{z}(2n+1)|^2) = -\frac{L(z)}{2}.
		\end{equation}
	\end{theorem}

	Since every GECMV matrix is gauge equivalent to an extended CMV matrix $\tilde{\mathcal{E}}$ via a diagonal unitary transformation (see Theorem~\ref{thm:rho_phases}), we can use the normalized Szeg\H{o} cocycle maps $S_{n,z}$ from \eqref{eq:szego_normalized} to compute the decay rate. 
	By \eqref{eq.ConjugatedTransferMatrix} and Theorem \ref{Lyapunov}, the corresponding Lyapunov exponent is $L(z)/2$.
	
	Let $\Psi^{z}$ be the solution to the eigenvalue equation of the extended CMV matrix 
	\begin{equation*}
		\tilde{\mathcal{E}}\Psi^{z}=z\Psi^{z}.
	\end{equation*}
	It is a general result \cite[Theorem 3.2]{craigSimon1983Duke} that 
	\begin{equation*}
		\liminf\limits_{|n|\to\infty} \frac{\log(|\Psi^{z}(2n)|^{2}+|\Psi^{z}(2n+1)|)}{2|n|}\geq -\frac{L(z)}{2}.
	\end{equation*}
	Note that the factor $\frac{1}{2}$ on the right-hand side is due to our specific way of defining the Lyapunov exponent. Thus we only need to give an upper bound to prove Theorem \ref{t:efctDecay}. Such an upper bound follows from the block-resolvent expansion of \cite{Jitomirskaya1999Annals}, with the power law growth of the scales counteracting the combinatorial factor. However, in the present context this expansion is more involved.
	
	To be more specific, let $\gamma=L(z)/2$ and $k_{5}$ be such that any $y$ with $|y|\geq k_{5}$ is $(\gamma-\epsilon,|y|)$- regular. To make the expansion clear, we use $i\in\bbN$ to track the times or levels of expansions, and $j\in\bbN$ to denote the index of endpoints of the interval in the definition of $(\gamma,k)$-regularity (Definition \ref{def.regular}). Let $n_{i,j}$ stand for the $j$-th endpoint of $i$-th level of expansion counting from left to right, let $I(x)$ be the interval containing the $(\gamma-\epsilon,|x|)$-regular $x$ given in the definition of regularity, and $n_{i,j}'$ be either $n_{i,j}$ or the interior neighbor of $n_{i,j}$ that belongs to the interval which takes $n_{i,j}$ as one of its endpoints.
	
	Denote $\tilde{G}_{I(x)}(x,\cdot)$ the Green's function defined in \eqref{eq.GreenFunction}. It is immediate to check that if $n_{i,j}$ is $(\gamma-\epsilon,|n_{i,j}|)$-regular, then it is contained in an interval $[n_{i+1,2j-1},n_{i+1,2j}]$ with $|n_{i,j}-n_{i+1,c_{i,j}}|\geq \frac{1}{7}|n_{i,j}|$, where $c_{i,j}=2j-1$ or $2j$ standing for the left or right boundary, respectively.
	Let $r>1$, we start from a $y$ large enough and  let $k_{6}$ be such that  $k_{6}^{r}<y<(k_{6}+3)^{r}$ and $k_{6}>k_{5}$. Since $y$ is $(\gamma-\epsilon,|y|)$-regular, there exists $[n_{1,1},n_{1,2}]$ containing $y$  and satisfies $k_{6}<\frac{1}{7}k_{6}^{r-1}k_{6}-1\leq\frac{1}{7}y-1\leq n_{1,1}<n_{1,2}$. Therefore, $n_{1,1}$ is $(\gamma-\epsilon,n_{1,1})$-regular, and there exists $[n_{2,1},n_{2,2}]$ containing $n_{1,1}$. We continue this expansion until either some $n'_{i,c_{i,j}}<k_{6}$ or the number of the $\tilde{G}_{I(n_{i,c_{i,j}}')}(n_{i,c_{i,j}}',n_{i+1,c_{i+1,j}})$ terms in the product in \eqref{eq.blockResolExp} below exceeds $7k_{6}^{r-1}$.
	
	This yields the following expansion for the generalized eigenfunction in \eqref{eq.solGreen}
	\begin{equation}\label{eq.blockResolExp}
		\Psi^{z}(y)=\sum_{s;j}\tilde{G}_{I(y)}(y,n_{1,j})\tilde{G}_{I(n_{1,j}')}(n_{1,j}',n_{2,c_{2,j}})\cdots \tilde{G}_{I(n_{s,c_{s,j}}')}(n_{s,c_{s,j}}',n_{s+1,c_{s+1,j}})\Psi^{z}(n'_{s+1,c_{s+1,j}})
	\end{equation}
	where $c_{i,j}=2j-1$ or $2j$, and each $n_{i,j}'$ can be specified by either $n_{i,j}$ or $n_{i,j}-(-1)^{j}$. By our design, we have $n'_{i,c_{i,j}}>k_{6}$ for $i=1,2,\cdots, s$ and either $n'_{s+1,c_{s+1,j}}<k_{6}$ and $s\leq 7k_{6}^{r-1}$, or $s+1=7k_{6}^{r-1}$.
	Note that in \eqref{eq.blockResolExp}, the $j$ in each $n'_{i+1,c_{i+1,j}}$ is indeed $j_{i+1}$, which stands for either the left or right end point of the interval containing $n'_{i,c_{i,j}}$, and needs not to be uniform for all $i=1,2,\cdots,s+1$.
	If $n'_{s+1,c_{s+1,j}}<k_{6}$ and $s\leq 7k_{6}^{r-1}$, we have 
	\begin{equation*}
		\begin{aligned}
			\big\vert \tilde{G}_{I(n)}(y,n_{1,j})\tilde{G}_{I(n_{1,j}')}(n_{1,j}',n_{2,c_{1,j}})\cdots \tilde{G}_{I(n_{s,c_{s,j}}')}&(n_{s,c_{s,j}}',n_{s+1,c_{s+1,j}})\Psi^{z}(n'_{s+1,c_{s+1,j}})\big\vert\\&\leq e^{-(\gamma-\epsilon)(|y-n'_{s+1,c_{s+1,j}}|+\sum_{1\leq i\leq s}|n_{i,c_{i,j}}'-n_{i,c_{i,j}}|)}\\
			&\leq e^{-(\gamma-\epsilon)(|y-n'_{s+1,c_{s+1,j}}|-(s+1))}\\
			&\leq e^{-(\gamma-\epsilon)(y-k_{6}-7k_{6}^{r-1})}.
		\end{aligned}
	\end{equation*}
	If $s+1=7k_{6}^{r-1}$, then $|y-n_{1j}|\geq\frac{k_{6}}{7},...,|n_{i,c_{i,j}}'-n_{i+1,c_{i+1,j}}|\geq\frac{k_{6}}{7}$, for $i=1,2,\cdots , s$, which yields the estimate
	\begin{equation*}\begin{aligned}
			\left\vert \tilde{G}_{I(y)}(y,n_{1,j})\tilde{G}_{I(n_{1,j}')}(n_{1,j}',n_{2,c_{2,j}})\cdots \tilde{G}_{I(n_{s,c_{s,j}}')}(n_{s,c_{s,j}}',n_{s+1,c_{s+1,j}})\Psi(n'_{s+1,c_{s+1,j}})\right\vert\leq e^{-(\gamma-\epsilon)\frac{k_{6}}{7}7k_{6}^{r-1}}.
		\end{aligned}
	\end{equation*}
	In both cases, we obtain 
	\begin{equation}\label{eq.prodUpB}
		\left\vert \tilde{G}_{I(y)}(y,n_{1,j})\tilde{G}_{I(n_{1,j}')}(n_{1,j}',n_{2,c_{2,j}})\cdots \tilde{G}_{I(n_{s,c_{s,j}}')}(n_{s,c_{s,j}}',n_{s+1,c_{s+1,j}})\Psi(n'_{s+1,c_{s+1,j}})\right\vert\leq e^{-(\gamma-\epsilon-\delta)|y|}
	\end{equation}
	for any $\delta>0$ and $y$ sufficiently large. The total number of terms in the sum can be bounded by $4^{7k_{6}^{r-1}}$. This together with \eqref{eq.prodUpB} gives
	\begin{equation*}
		|\Psi^{z}(y)|\leq 4^{7|y|^{\frac{r-1}{r}}}e^{-(\frac{L(z)}{2}-\epsilon-\delta)|y|}.
	\end{equation*}
	Since $\delta,\epsilon$ can be arbitrarily small, the upper bound is therefore obtained.
	\hfill\qedsymbol

	\section{Absolutely Continuous Spectrum in the Subcritical Regime}\label{sec:subcriticalAc}
	
	According to \cite[Appendix B]{LDZ2022TAMS}, the spectral measure of an extended CMV matrix with quasi-periodic Verblunsky coefficients is purely absolutely continuous continuous in the {\it subcritical} region.
	A key ingredient of \cite{LDZ2022TAMS} is the analysis of the Szeg\H{o} cocycle of the extended CMV matrix. Namely, if  one has relatively ``good" control on the growth of these cocycle, then the spectral measure can likewise be controlled. However, this argument does not apply directly to the mosaic UAMO with coefficients \eqref{eq:mosaicVerblunskies} for two reasons: Firstly, the doubly-infinite matrix corresponding to the mosaic UAMO is not a standard extended CMV matrix, since the $\rho$'s appearing therein are complex whenever $\lambda_2\neq1$, see \eqref{eq:mosaicVerblunskies}. Secondly, the coefficients in \eqref{eq:mosaicVerblunskies} are not quasi-periodic but merely almost-periodic.
	
	Fortunately, we can do away with the first obstacle by appealing to the gauge transform $D$ in Theorem~\ref{thm:rho_phases}. The second obstacle can be resolved by considering the four-step combined Szeg\H{o} cocycle $(2\Phi,S^{++}_{e^{it}})$, which is quasi-periodic. Theorem \ref{Lyapunov} guarantees that $(2\Phi,S^{++}_{e^{it}})$ is subcritical for $e^{it}\in\Sigma$ whenever $F(\lambda_{1},\lambda_{2},t)<0$, and that the behavior of the solution to the eigenvalue equation of the mosaic UAMO is the same as that of the corresponding extended CMV matrix with real $\rho$'s when the phase $\theta$ is complexified. 
	
	We need the following global-to-local reduction lemma to turn the subcritical cocycles into perturbations of constant ones. 
	Let $h>0$ be given, and for any function $f$ defined on $\{z\in\bbC:|\Im z|<h\}$, let $\Vert f\Vert_{h}=\sup_{|\Im z|<h}\Vert f\Vert$.
	\begin{lemma}\label{ARC}
		Let $\Phi\in \DC$, and let $\Sigma^{sub}$ be the set of spectral parameters for which the cocycle $(2\Phi,S^{++}_{e^{it}})$ is subcritical. Then there exists $h=h(2\Phi)>0$ such that for any $\eta>0$, and for any $e^{it}\in\Sigma^{sub}$, there exist $Z_{t}\in C^{\omega}(2\bbT,\SU(1,1))$, $f_{t}(\theta)\in C^{\omega}(\bbT,\mathrm{su}(1,1))$, $\phi(t)\in\bbR$, $D_{\phi(t)}=\mathrm{diag}\{e^{i\phi(t)},e^{-i\phi(t)}\}$ such that
		\begin{equation}\label{eq.conjugate}
			Z^{-1}_{t}(\theta+2\Phi)S^{++}_{e^{it}}(\theta)Z_{t}(\theta)=D_{\phi(t)}e^{f_{t}(\theta)}
		\end{equation}
		with $\|f_{t}\|_{h}<\eta,\|Z_{t}\|_{h}<\Gamma(\Phi,\eta,\lambda_{1},\lambda_{2})$ for some constant $\Gamma$.
	\end{lemma}
	\begin{remark}
		The proof depends on Avila's solution of almost-reducible conjecture \cite{Avila2015Acta,arc}, that is, if the cocycle is subcritical then it is almost-reducible.  With the compactness argument from \cite[Proposition 5.2]{LYZZ}, the key observation here is that one can choose $h$ and $\Gamma(\Phi,\eta)$ to be independent of $e^{it}$; see also \cite[Lemma 4.2]{wangExactMobilityEdges2021a}.
	\end{remark}
	
	Combining this lemma with the main results of \cite{LDZ2022TAMS}, we can conclude that the spectral measure is purely absolutely continuous on $\Sigma^{sub}$, that is for $t<t_{0}(\lambda_{1},\lambda_{2})$ and $2\pi-t_{0}<t<2\pi$ with $t_0$ as defined in \eqref{eq:t_0}.
	We give the sketch of proving purely absolutely continuous spectrum for the reader's convenience, and direct the reader to \cite{LDZ2022TAMS} for a more detailed proof.
	
	In order to facilitate our statement, we need to introduce the prescriptions of notations that will be needed:
	Following Lemma \ref{ARC}, let $\epsilon_{0}=\eta$ be sufficiently small and define the sequences
	\begin{equation*}
		\epsilon_{j} = \epsilon_{0}^{2^{j}}, \qquad r_{j} = \frac{r}{2^{j}}, \qquad N_{j} = \frac{4^{j+1} \log \epsilon_{0}^{-1}}{r}
	\end{equation*}
	as each standard KAM argument does. Let $\rho(2\Phi,D_{\phi(t)}e^{f_{t}})$ be the fibered rotation number of the cocycle. Then we have the following result as an application of \cite[Proposition 3.1]{LDZ2022TAMS} to the near constant cocycle $(2\Phi,D_{\phi(t)}e^{f_{t}})$:

	\begin{theorem}\label{rd1}
		Assume that $\kappa, \tau, r > 0$ and $\Phi \in \DC(\kappa,\tau)$. Let $D_{\phi(t)}\in \SU(1,1)$, $f_{t} \in C^{\omega}_{r}(\bbT^{d}, \mathrm{su}(1,1))$ with
		\begin{equation*}
			\Vert f_{t} \Vert_{r} \leq \epsilon_{0} \leq \frac{D_{0}}{\left\| D_{\phi(t)}\right\|^{C_{0}}} \left( \frac{r}{2} \right)^{C_{0}\tau},
		\end{equation*}
		where $D_{0} = D_0(\kappa,\tau)$ and $C_{0}$ is a numerical constant. Then for any $j \geq 1$, there exists $B_{j} \in C^{\omega}_{r_{j}}(2\bbT^{d}, \SU(1,1))$ such that
		\begin{equation*}
			B_{j}(\theta+2\Phi)(D_{\phi(t)}e^{f_{t}(\theta)}) B^{-1}_{j}(\theta) = D_{\phi(t)}^{j} e^{f_{t}^{j}(\theta)},
		\end{equation*}
		where $\Vert f_{t}^{j}(\theta) \Vert_{r_{j}}\leq \epsilon_{j}$ and $B_{j}$ satisfies
		\begin{align}
			\label{normOfB}
			\left\| B_{j}\right\|_{0} & \leq \epsilon_{j-1}^{-\frac{1}{192}}, \\
			\label{degreeOfB}\vert \deg{B_{j}}\vert & \leq 2N_{j-1}.
		\end{align}
		More precisely, we have
		
		{\rm (a)}
		If  $ \Vert 2 \rho( \Phi, D_{\phi(t)}^{j-1}e^{f_{t}^{j-1}})- \langle m, 2\Phi \rangle \Vert_{\bbR/\bbZ} < \epsilon_{j-1}^{\frac{1}{15}}$ for some $m\in \bbZ^{d}$ with  $0 < \vert m \vert < N_{j-1}$ , we have the following precise expression:
		\begin{equation*}
			D_{\phi(t)}^{j} = \exp \begin{bmatrix} i t_{j} & v_{j} \\ \overline{v_{j}} & -i t_{j} \end{bmatrix},
		\end{equation*}
		where $t_{j} \in \bbR$, $v_{j} \in \bbC$, and $\vert t_{j} \vert \leq \epsilon_{j-1}^{\frac{1}{16}}$, $\vert v_{j} \vert \leq \epsilon_{j-1}^{\frac{15}{16}}$.
		
		{\rm (b)}  Moreover, there always exist unitary matrices $U_{j}\in \SL(2,\bbC)$ such that
		\begin{equation}\label{e.conj}
			U_{j} D_{\phi(t)}^{j} e^{f_{t}^{j}(x)} U_{j}^{-1} = \begin{bmatrix} e^{2\pi i\rho_{j}} & c_{j} \\ 0 & e^{-2\pi i\rho_{j}} \end{bmatrix} + F_{t}^{j}(x)
		\end{equation}
		where $\rho_{j} \in \bbR \cup i \bbR$, with estimates $\Vert F_{t}^{j} \Vert_{r_{j}} \leq \epsilon_{j}$, and
		\begin{equation}\label{Btimesc}
			\Vert B_{j} \Vert_{0}^{2} \vert c_{j} \vert \leq 8 \Vert D_{\phi(t)} \Vert.
		\end{equation}
	\end{theorem}
	Curious readers may consult \cite{LDZ2022TAMS,wangExactMobilityEdges2021a} for the proof. Note that the cocycle  $(2\Phi,S^{++}_{e^{it}})$
	represents a four-steps combined iteration. A simple observation is that 
	\begin{equation*}
		S_{4n,z}=S_{4n+2,z}=\frac{1}{\lambda_{1}}\begin{bmatrix}z^{\frac{1}{2}}&-\lambda_{1}'z^{-\frac{1}{2}}\\-\lambda_{1}'z^{\frac{1}{2}}&z^{-\frac{1}{2}}\end{bmatrix},\qquad S_{4n+1,z}=\begin{bmatrix}z^{\frac{1}{2}}&0\\0&z^{-\frac{1}{2}}\end{bmatrix}.
	\end{equation*}
	Let $T^{n}_{z}=\prod_{j=n-1}^{0}S_{j,z}$ be the transfer matrix of the normalized Szeg\H{o} cocycle maps. It follows immediately that there exists a positive constant $C=C(\lambda_{1})$ such that 
	\begin{equation}\label{eq.ulbound}
		C^{-1}\left\Vert \prod_{j=n-1}^{0}D_{\phi(t)}e^{f_{t}(\theta+2j\Phi)}\right\Vert\leq\Vert T^{4n+k}_{e^{it}}\Vert\leq C\left\Vert \prod_{j=n-1}^{0}D_{\phi(t)}e^{f_{t}(\theta+2j\Phi)}\right\Vert \quad\text{for}\quad k=0,1,2,3.
	\end{equation}
	This enables us to translate the estimates for  $(2\Phi,D_{\phi(t)}e^{f_{t}})$ (and thus  $(2\Phi,S^{++}_{e^{it}})$ by Lemma \ref{ARC}) given by Theorem \ref{rd1} to  the corresponding estimates of the transfer matrix of Szeg\H{o} cocycle maps of any length.

	For any $m\in\bbZ^{d}$ with $0<|m|<N_{j-1}$, define \begin{equation}
		\label{e.lambdamjdef}
		\Lambda_{m}(j) = \left\{ e^{it} \in \Sigma : \left\Vert 2 \rho\left( 2\Phi, D_{\phi(t)}^{j-1}e^{f^{j-1}_{t}}\right)- \langle m, 2\Phi \rangle \right\Vert_{\bbR/\bbZ} < \epsilon_{j-1}^{\frac{1}{15}}  \right\},
	\end{equation}
	and
	\begin{equation}\label{e.Kjdef}
		K_{j}=\bigcup_{0<\vert m\vert\leq N_{j-1}}\Lambda_{m}(j)
	\end{equation}
	with $\Lambda_{m}(j)$ from \eqref{e.lambdamjdef}. 
	Let 
	\begin{equation*}
		F(z) = \int \frac{e^{i\theta}+z}{e^{i\theta}-z} \, d\mu
	\end{equation*}
	be the Carath\'eodory function of a measure $\mu$, then following the CMV version of the Damanik-Killip-Lenz maximum modulus principle argument \cite{DamKilLen2000CMP} of Munger-Ong \cite{MungerOng2014JMP}, gives
	\begin{equation}\label{eq.caraUpB}
		\Re f((1-\epsilon)e^{it}) \geq \frac{1}{\epsilon}\mu(t-\epsilon,t+\epsilon)
	\end{equation}
	for any $t$ and $\epsilon>0$ small. Together with the Jitomirskaya-Last inequality of the CMV version (see Section 10.8 of \cite{Simon2005OPUC2}), we have
	\begin{equation}\label{eq.measControl}
		\mu(t-\epsilon,t+\epsilon)<C \epsilon\sup_{0\leq s\leq c\epsilon^{-1}}\Vert T^{s}_{e^{it}}\Vert^{2}.
	\end{equation}

	As a well known result, let $\mathcal{B}=\{t\in[0,2\pi]: \limsup_{s}\Vert T^{s}_{e^{it}}\Vert_{0}<\infty\}$, then $\mu\vert_{\mathcal{B}}$ is absolutely continuous. Therefore, absolute continuity of $\mu$ follows from $\mu(\Sigma\backslash\mathcal{B})=0$. 
	Let $\mathcal{R}$ denote the collection of the spectral parameters for which $(2\Phi,D_{\phi(t)}e^{f_{t}})$ is reducible, since $\mathcal{R}\backslash\mathcal{B}$ is the set of spectral parameters for which the cocycle is reducible to the parabolic, thus at most countable and supports no point spectrum, $\mu(\mathcal{R}\backslash\mathcal{B})=0$. Therefore it suffices to show $\mu(\Sigma\backslash\mathcal{R})=0$. To this end, we need the observation that by our construction of $K_{j}$ in \eqref{e.lambdamjdef} and \eqref{e.Kjdef}, we have $\Sigma\backslash\mathcal{R}\subset\limsup K_{j}$. That is, irreducible spectral parameters of $\Sigma$ belong to infinitely many $K_{j}$'s. On each $K_{j}$, combining estimates of Theorem \ref{rd1} and \eqref{eq.measControl}, the following inequality holds
	\begin{equation*}
		\mu(K_{j})\leq C\epsilon_{j-1}^{\frac{7}{384}}
	\end{equation*}
	which implies that $\sum_{j}\mu(\overline{K_{j}})<\infty$. By the {\it Borel-Cantelli Lemma}, $\mu(\Sigma\backslash \mathcal{R})=0$, which finishes the proof.

	\begin{appendix}
		\section{Proof of Lemma \ref{uniformlyDistri}}\label{uniform}

		We first need the following result of \cite{AJ09}:
		\begin{lemma}\label{lem.AJ09}
			Let $\omega\in\bbR\setminus\bbQ$, $x\in\bbR$, and $q_{m}$ be the denominator of continued fraction approximants of $\omega$. 
			Let $0\leq l_{0}\leq q_{m}-1$ be such that 
			\begin{equation*}
				|\sin\pi(x+l_{0}\omega)|=\inf_{0\leq l\leq q_{m}-1}|\sin\pi(x+l\omega)|,
			\end{equation*}   
			then for some absolute constant $C$,
			\begin{equation}\label{eq.AJEstimate}
				-C\log q_{m}\leq \sum_{0\leq l\leq q_{m}-1,l\neq l_{0}}\log|\sin\pi(\theta+l\omega)|+(q_{m}-1)\log 2\leq C\log q_{m}.
			\end{equation}
		\end{lemma}
		\medskip
		Let $z=\cos2\pi a\in[-1,1]$, our goal is to obtain the estimate: \begin{equation*}
			\sum_{j\in I_{1}\cup I_{2},j\neq i}\left(\log |\cos2\pi a-\cos2\pi\theta_{j}|-\log|\cos2\pi\theta_{i}-\cos2\pi\theta_{j}|\right)<3sq_{m}\epsilon
		\end{equation*}
		for any $i$.
		Denote 
		\begin{equation*}
			S_{1}=\sum\limits_{j\in I_{1}\cup I_{2},j\neq i}\log|\cos2\pi a-\cos2\pi\theta_{j}|+3s(q_{m}-1)\log 2
		\end{equation*}    
		and 
		\begin{equation*}
			S_{2}=\sum\limits_{j\in I_{1}\cup I_{2},j\neq i}\log|\cos2\pi\theta_{i}-\cos2\pi\theta_{j}|+3s(q_{m}-1)\log 2.
		\end{equation*}    
		By a trigonometric identity, 
		\begin{equation*}
			S_{1}=\left(\sum\limits_{j\in I_{1}\cup I_{2},j\neq i}\log|\sin\pi(a+\theta_{j})|+\log|\sin\pi(a-\theta_{j})|\right)+(3sq_{m}-1)\log2.
		\end{equation*}
		Note that the sum in $S_{1}$ contains $3sq_{m}$ terms, which we can divide it into $3s$ groups, each of which contains $q_{m}$ terms and then apply Lemma \ref{lem.AJ09}. We have the following:
		\begin{equation}\label{eq.estS1}S_{1}\leq -3s(q_{m}-1)\log2+3sC\log q_{m}.\end{equation}
		Similarly, we can write 
		\begin{equation*}
			S_{2}=\left(\sum\limits_{j\in I_{1}\cup I_{2},j\neq i}\log|\sin\pi(2\theta+(i+j)2\Phi)|+\log|\sin\pi(i-j)2\Phi|\right)+3s(q_{m}-1)\log2.
		\end{equation*}
		
		Since $\Phi\in\DC(\kappa,\tau)$, for any $0<|j|< q_{m+1}$, we have 
		\begin{equation*}
			\Vert j2\Phi\Vert_{\bbT}\geq \Vert q_{m}2\Phi\Vert_{\bbT}\geq\frac{\kappa}{(2q_{m})^{\tau}},
		\end{equation*}    
		which implies that \begin{equation}\label{eq.dioEq1}\max\{\log|\sin\pi x|,\log|\sin\pi(x+j2\Phi)|\}\geq 2\log\kappa-2\tau\log 2q_{m}.\end{equation}
		Since $\theta$ is non-resonant with respect to $\Phi$, we have 
		\begin{equation}\label{eq.dioEq2}
			\log|\sin2\pi(\theta+(i+j)\Phi)|\geq -|i+j|^{\frac{1}{2\tau}}\geq -(20sq_{m})^{\frac{1}{2\tau}},
		\end{equation}
		and since $\Phi\in\DC(\kappa,\tau)$, we also have 
		\begin{equation}\label{eq.dioEq3}
			\log|\sin\pi(i-j)2\Phi|\geq \log\kappa-\tau\log20sq_{m}.
		\end{equation}
		By \eqref{eq.dioEq1}, every $sq_{m}$ terms in the sum of $S_{2}$ may contain at most one extra small term that is bounded from below by \eqref{eq.dioEq2} and \eqref{eq.dioEq3}.

		Therefore, we have the following estimate for $S_{2}$:
		\begin{equation}\label{eq.estS2}S_{2}\geq -6sC\log q_{m}-3s(q_{m}-1)\log2-3(20sq_{m})^{\frac{1}{2\tau}}+3(\log\kappa-\tau\log20sq_{m}).
		\end{equation}
		Combining \eqref{eq.estS1} with \eqref{eq.estS2} gives
		\begin{equation*}
			S_{1}-S_{2}\leq 3sq_{m}\epsilon=k\epsilon.
		\end{equation*}
		\qed

		\section{S(uper)GECMV matrices}\label{sec:supergecmv}
		
		Theorem~\ref{thm:rho_phases} generalizes to Verblunsky pairs whose vector 2-norm is a phase in the following way: Consider the GECMV matrix $\mathcal E=\mathcal E_{\alpha,\rho}$ as defined in Section \ref{sec:results}, but with the $\Theta$-matrices specified by Verblunsky pairs $(\alpha,\rho)$ satisfying the relaxed condition
		\begin{equation*}
			|\alpha|^2+|\rho|^2=e^{2i\varphi}=-\det\Theta(\alpha,\rho).
		\end{equation*}
		Then
		\begin{prop}
			$\mathcal E$ is isospectral to a standard extended CMV matrix.
		\end{prop}
		\begin{proof}
			The proof goes along the same lines as that of Theorem~\ref{thm:rho_phases}: 
			Fix $d_0,d_{-1} \in \partial \bbD$ and define the entries of $D$ recursively by
			\begin{equation}\label{eq:super_lambdas}
				d_{2n+2}=\xi_{2n+1}^{-1}\xi_{2n}^{-1}e^{-i(\varphi_{2n+1}+\varphi_{2n})}d_{2n},\qquad d_{2n+1}=\xi_{2n-1}^{-1}\xi_{2n}^{-1}e^{-i(\varphi_{2n}+\varphi_{2n-1})}d_{2n-1},
			\end{equation}
			where $\rho=\xi|\rho|$. 
			We then define the new Verblunsky coefficients
			\begin{align}
				\label{eq:newalphasSGECMV1}
				\tilde\alpha_{2n-1}&=\left[\prod_{k=0}^{n-1}e^{-i(\varphi_{2k+1}+2\varphi_{2k}+\varphi_{2k-1})}\right]\xi_{-1}\frac{d_{0}}{d_{-1}}\alpha_{2n-1},\\  
				\label{eq:newalphasSGECMV2}\tilde\alpha_{2n}&=\left[\prod_{k=1}^{n}e^{-i(\varphi_{2k}+2\varphi_{2k-1}+\varphi_{2k-2})}\right]e^{-i(\varphi_{-1}+\varphi_{0})/2}\xi_{-1}\frac{d_{0}}{d_{-1}}\alpha_{2n},\\
				\tilde\rho_{k}&=|\rho_{k}|,
			\end{align}
			and denote by $\tilde{\mathcal E}$ the extended CMV matrix corresponding to $\tilde\alpha$ and $\tilde\rho$. To conclude, we will  demonstrate
			\begin{equation}\label{eq:ERealifiedGEN}
				\tilde{\mathcal E}=D^*\mathcal E D.
			\end{equation}
			From the recursion relation \eqref{eq:super_lambdas} we get
			\begin{align*}
				\frac{d_{2n}}{d_{2n-1}}&=\prod_{k=0}^{n-1}e^{-i(\varphi_{2k+1}+2\varphi_{2k}+\varphi_{2k-1})}\frac{\xi_{-1}}{\xi_{2n-1}}\frac{d_0}{d_{-1}}, \\
				\frac{d_{2n}}{d_{2n+1}}&=\prod_{k=1}^{n}e^{-i(\varphi_{2k}+2\varphi_{2k-1}+\varphi_{2k-2})}e^{-i(\varphi_{-1}+\varphi_{0})}\xi_{2n}\xi_{-1}\frac{d_0}{d_{-1}}.
			\end{align*}
			We then calculate that for all integers $n$,
			\begin{align*}
				\overline{d_{2n}}d_{2n+2}e^{i(\varphi_{2n+1}+\varphi_{2n})}\rho_{2n+1}\rho_{2n}&=\tilde\rho_{2n+1}\tilde\rho_{2n},\\
				\overline{d_{2n+1}}d_{2n-1}e^{i(\varphi_{2n-1}+\varphi_{2n})}\overline{\rho_{2n-1}\rho_{2n}}&=\overline{\tilde{\rho}_{2n-1}}\overline{\tilde{\rho}_{2n}} \\
				d_{2n+2}\overline{d_{2n+1}}e^{i(\varphi_{2n+1}+\varphi_{2n})}\rho_{2n+1}\alpha_{2n}  &= \tilde\rho_{2n+1}\tilde\alpha_{2n}\\
				d_{2n}\overline{d_{2n+1}}e^{i(\varphi_{2n-1}+\varphi_{2n})}\overline{\rho_{2n}}\alpha_{2n-1}&=\overline{\tilde\rho_{2n}}\tilde\alpha_{2n-1}.
			\end{align*}
			This suffices to prove \eqref{eq:ERealifiedGEN}.
		\end{proof}
		
		\begin{remark}
			Notice that if the phases $\varphi$ are nontrivial, then \eqref{eq:newalphasSGECMV1} and \eqref{eq:newalphasSGECMV2} show that one cannot in general hope to preserve that $\alpha$'s under the gauge transform here; compare Remark~\ref{rem:gecmvWarning}.
		\end{remark}
		
	\end{appendix}

	\bibliographystyle{abbrvArXiv}
	
	\bibliography{mosaic-bib}
	
\end{document}